\documentclass[11pt]{amsart}

\usepackage{amscd,verbatim}
\usepackage{amsmath, amssymb}
\usepackage{graphicx,color,hyperref}
\usepackage{amsfonts}
\newcommand{\de}{\partial}

\newcommand{\ddbar}{i \partial \overline{\partial}}
\newcommand{\Ric}{\mathrm{Ric}}
\newcommand{\ov}[1]{\overline{#1}}

\newcommand{\ti}[1]{\tilde{#1}}
\newcommand{\vp}{\varphi}

\newcommand{\ve}{\varepsilon}

\renewcommand{\leq}{\leqslant}
\renewcommand{\geq}{\geqslant}

\newcommand{\be}{\begin{equation}}
\newcommand{\ee}{\end{equation}}

\begin{document}
\newcounter{remark}
\newcounter{theor}
\setcounter{remark}{0}
\setcounter{theor}{1}
\newtheorem{claim}{Claim}
\newtheorem{theorem}{Theorem}[section]
\newtheorem{lemma}[theorem]{Lemma}
\newtheorem{corollary}[theorem]{Corollary}
\newtheorem{conjecture}[theorem]{Conjecture}
\newtheorem{proposition}[theorem]{Proposition}
\newtheorem{question}{question}[section]
\newtheorem{defn}{Definition}[theor]
\theoremstyle{definition}
\newtheorem{rmk}[theorem]{Remark}

\newenvironment{example}[1][Example]{\addtocounter{remark}{1} \begin{trivlist}
\item[\hskip
\labelsep {\bfseries #1  \thesection.\theremark}]}{\end{trivlist}}

\title[Special K\"ahler geometry and Lagrangian fibrations]{Special K\"ahler geometry and holomorphic Lagrangian fibrations}

\author{Yang Li}
\address{Massachusetts Institute of Technology, 77 Massachusetts Avenue, Cambridge, MA 02139}
\email{yangmit@mit.edu}
\author{Valentino Tosatti}
\address{Courant Institute of Mathematical Sciences, New York University, 251 Mercer St, New York, NY 10012}
\email{tosatti@cims.nyu.edu}
\dedicatory{Dedicated to the memory of Jean-Pierre Demailly}

\begin{abstract}
Given a holomorphic Lagrangian fibration of a compact hyperk\"ahler manifold, we use the differential geometry of the special K\"ahler metric that exists on the base away from the discriminant locus, and show that the pullback of the tangent bundle of the base to the total space of a family of minimal rational curves admits a parallel splitting. The splitting is nontrivial when the base is not half-dimensional projective space. Combining this with results of Voisin, Hwang and Bakker-Schnell, we deduce that the base must be  projective space, a result first proved by Hwang.
\end{abstract}

\maketitle

\section{Introduction}

Let $X^{2n}$ be a hyperk\"ahler manifold, so $X$ is a simply connected compact K\"ahler manifold with a holomorphic symplectic $2$-form $\Omega$ such that $H^{2,0}(X)=\mathbb{C}\Omega$. By Yau's Theorem every K\"ahler class on $X$ contains a unique Ricci-flat K\"ahler metric. It was later realized by Beauville \cite{Be} that these metrics are hyperk\"ahler, which means that they have holonomy equal to $Sp(n)$.

Suppose that $B$ is an irreducible normal complex analytic space with $0<\dim B<2n$, and $f:X\to B$ is a holomorphic surjective map with connected fibers. Then work of Matsushita \cite{Ma} shows that necessarily $\dim B=n$, that all irreducible components of the fibers of $f$ are Lagrangian with respect to $\Omega$, and the smooth fibers are tori. We call such $f$ a holomorphic Lagrangian fibration. The following basic conjecture is widely expected to hold:

\begin{conjecture}\label{co}
If $X$ is a hyperk\"ahler manifold and $f:X\to B$ is a holomorphic Lagrangian fibration, then $B\cong\mathbb{P}^n$.
\end{conjecture}

This conjecture is clearly true when $n=1$.  The most striking result about this Conjecture is due to Hwang \cite{Hw}:

\begin{theorem}[Hwang \cite{Hw}]\label{main}
Conjecture \ref{co} holds if $X$ is projective and $B$ is smooth.
\end{theorem}

Theorem \ref{main} was later extended to $X$ K\"ahler and $B$ smooth by Greb-Lehn \cite{GL}. Assuming $X$ projective and $n=2$, it was proved by Ou \cite{Ou} that either $B$ is smooth (hence $\mathbb{P}^2$) or else it has just one very specific singular point. This case was later ruled out independently by Bogomolov-Kurnosov \cite{BK} and Huybrechts-Xu \cite{HX}, so the conjecture is known in this case. It is also known for some families of hyperk\"ahler manifolds \cite{BM,CMSB,Mar,Ma4,Yos}, but it remains open in general.

There are also a number of partial results towards Conjecture \ref{co} in general, see \cite{HM} for an excellent recent overview. It is known that $B$ must be a K\"ahler space (see e.g. \cite[Prop. 2.2]{GL}) and Moishezon \cite[\S 2.3]{Ma3}, and that $B$  is $\mathbb{Q}$-factorial and has at worst klt singularities (by \cite[Theorem 2.1]{Ma3}). It follows that $B$ has at worst rational singularities, and hence it is projective by \cite[Corollary 1.7]{Nam}. Again thanks to \cite[Theorem 2.1]{Ma3} we see that $B$ is a Fano variety with Picard number one, and in particular it is uniruled \cite{MM} and simply connected \cite{Taka}. The rational cohomology of $B$ is isomorphic to the one of $\mathbb{P}^n$ \cite{SY}. It is also known that the map $f$ is locally projective \cite{Cam}, so the smooth fibers are abelian varieties, and if $B$ is smooth then the discriminant locus $D\subset B$ of $f$ has pure codimension $1$ by \cite[Proposition 3.1]{HO}.

Our main result forms part of a new proof of Hwang's theorem, as well as Greb-Lehn's extension. In order to describe this, suppose $B$ is not $\mathbb{P}^n$. Then from a result of Cho-Miyaoka-Shepherd-Barron \cite{CMSB}, which uses Mori theory, it follows that there is a rational curve in $B$ (not contained in $D$) with anticanonical degree at most $n$. We show that such a curve is free, and together these imply that the Grothendieck decomposition of the pullback of $TB$ to this rational curve has some degree zero factors. Taking such rational curves with minimal anticanonical degree, we can consider the universal family $\mathcal{U}$ with evaluation map $\mu:\mathcal{U}\to B$, which we may assume is a submersion over a Zariski open set $B^\circ\subset B$ (which we may assume is equal to $B\backslash D$ up to enlarging $D$), and the positive degree factors in the Grothendieck decomposition define a nontrivial holomorphic subbundle $\mathcal{V}\subset \mu^*TB^\circ$, whose rank is strictly less than $n$. At the same time, classical work of Freed \cite{Fr} shows that on $B^\circ$ there is a ``special K\"ahler metric'' $g_{\rm SK}$, whose K\"ahler form $\omega_{\rm SK}$ is parallel with respect to a ``special K\"ahler connection'' $\nabla^{\rm SK}$ on $T^{\mathbb{R}}B^\circ$, which is torsion-free, flat, and $d^{\nabla^{\rm SK}}J=0$ (where $J$ is the complex structure of $B$). Our main result is then:

\begin{theorem}\label{mainn}
In this setting, $\mathcal{V}$ is preserved by the pullback of the Chern connection of $g_{\rm SK}$.
\end{theorem}

We also show that the corresponding real subbundle $\mathcal{V}_{\mathbb{R}}\subset \mu^*T^{\mathbb{R}}B^\circ$ is also preserved by the pullback of the special K\"ahler connection $\nabla^{\rm SK}$. This in turn can be interpreted as giving a nontrivial splitting of a real variation of Hodge structures (which naturally exists on $B^\circ$) when pulled back via $\mu$. As we will discuss below, by combining Theorem \ref{mainn} with work of Voisin \cite{Vo}, Hwang \cite{Hw2,Hw} and Bakker-Schnell \cite{BS}, one can deduce Hwang's Theorem \ref{main}.

Let us first give some intuition for our approach. One of the key features of the rich geometry of special K\"ahler metrics is that they have nonnegative bisectional curvature. Recall here the fundamental theorem of Mori \cite{Mor} and Siu-Yau \cite{SY2} which states that a {\em compact} K\"ahler manifold with positive bisectional curvature must be isomorphic to $\mathbb{P}^n$. This was generalized by Mok \cite{Mok} to classify compact K\"ahler manifolds with nonnegative bisectional curvature: their universal cover splits as a product of a Euclidean factor, of projective space, and of compact Hermitian symmetric spaces of rank $\geq 2$. A large part of our arguments are motivated by trying to extend Mok's techniques to our {\em noncompact} manifold $B\backslash D$ with an incomplete metric with nonnegative bisectional curvature, making essential use of the special features of special K\"ahler metrics, which are summarized in Section \ref{se2}.

To prove Theorem \ref{mainn}, thanks to a recent result of Bakker \cite{Ba} we need to consider two cases: either $f$ has maximal variation or $f$ is isotrivial. In the first case, we prove in Section \ref{se4} a crucial rigidity result (Theorem \ref{rigid}) which shows that the bisectional curvature of $\omega_{\rm SK}$  vanishes when evaluated on a vector in $\mathcal{V}$ and a vector in its orthogonal complement. For this, we use results of Zhang and the second-named author \cite{TZ} on the asymptotic behavior of $\omega_{\rm SK}$ near $D$, as well as a strictly positive lower bound for $\omega_{\rm SK}$  near $D$ obtained by Gross, Zhang and the second-named author in \cite{GTZ,GTZ2,To}. These are explained in Section \ref{se3}. In Section \ref{se5} we then supplement the rigidity result by showing that the rough Laplacian of the bisectional curvature of $\omega_{\rm SK}$  evaluated on the same vectors vanishes as well. This result is analogous to a statement in Mok \cite{Mok}, although our proof is quite different. Equipped with these rigidity results, in Section \ref{se6} we adapt an argument of Mok \cite{Mok} and conclude.  In the isotrivial case the rigidity results are trivial because $\omega_{\rm SK}$ is flat, but this flatness can be effectively exploited to show again that $\mathcal{V}$ is preserved by the Chern connection of $g_{\rm SK}$.

In section \ref{se7} we sketch how Theorem \ref{main} follows by combining Theorem \ref{mainn} with a number of recent results in the literature.  As mentioned above, we first show that the real subbundle $\mathcal{V}_{\mathbb{R}}\subset \mu^*T^{\mathbb{R}}B^\circ$ which corresponds to $\mathcal{V}$ is preserved also by the pullback of the special K\"ahler connection $\nabla^{\rm SK}$. This uses again our rigidity theorem. Then we invoke an important result of Hwang \cite{Hw2,Hw}, which also has a recent proof by Bakker-Schnell \cite{BS} (Theorem \ref{bs} below), which gives that the map $\mu$ must have connected fibers. Thus, our splitting descends to a parallel splitting of $T^{\mathbb{R}}B^\circ$, from which we obtain a parallel real $(1,1)$-form on $B^\circ$ which is not proportional to $\omega_{\rm SK}$, which is contradiction to a result of Voisin \cite{Vo}.

Lastly, in Section \ref{se8} we make some comments on the obstacles that we faced when trying to extend our approach to the case when $B$ is singular.\\

\begin{rmk}
In the first draft of our paper, our original argument in Section \ref{se7} to construct the parallel form on $B^\circ$ turned out to be incomplete. After our first draft was posted to arXiv, Bakker and Schnell sent us their paper \cite{BS} with a new proof of Hwang's theorem. As mentioned above, to deduce Theorem \ref{main} from Theorem \ref{mainn} we now rely on their paper. On the other hand, without using \cite{BS}, what our arguments show is that $B$ must be $\mathbb{P}^n$ provided that $\mu$ has connected fibers. As pointed out to us by Hwang, this result was implictly proved by Cho-Miyaoka-Shepherd-Barron \cite[\S 7]{CMSB} using a different method (under the extra assumption that $f$ has a section, which was removed by Nagai \cite{Nag}).
\end{rmk}

\subsection{Acknowledgments} This research was conducted during the period when the first-named author served as a Clay Research Fellow. The second-named author was partially supported by NSF grant DMS-2231783, and part of this work was carried out during his visit to the Center for Mathematical Sciences and Applications at Harvard University, which he would like to thank for the hospitality. We are grateful to B. Bakker, J. Cao, J. Koll\'ar, M. P\u{a}un, T. Peternell, C. Schnell and J. Starr for enlightening email communications, and especially to B. Bakker, J.-M. Hwang, C. Schnell, C. Voisin and the referee for very useful comments on previous drafts. The second-named author would also like to thank Y. Zhang for discussions about this approach during the writing of \cite{TZ}. This work is dedicated to the memory of Jean-Pierre Demailly, a giant in the field of complex geometry and a dear colleague and friend, who left us much too early.

\section{Special K\"ahler metrics}\label{se2}
\subsection{Notation} Let us first fix some notation. For a complex manifold $B$ we will denote by $T^{\mathbb{R}}B$ its real tangent bundle, and with $TB\subset T^{\mathbb{C}}B=T^{\mathbb{R}}B\otimes\mathbb{C}$ its holomorphic tangent bundle (of complex tangent vectors of type $(1,0)$). The dual of $TB$ will be denoted by $\Omega^1_B$. The complex structure will be denoted by $J:T^{\mathbb{R}}B\to T^{\mathbb{R}}B$.  We will also denote $B^\circ:=B\backslash D$ and $X^\circ:=f^{-1}(B^\circ)$.

\subsection{Existence of special K\"ahler metrics}
The paper by Freed \cite{Fr}, following work of Donagi-Witten \cite{DW}, shows that the base of an algebraic integrable system (which in our case is $B^\circ$) admits a geometric structure called ``special K\"ahler metric'', $\omega_{\rm SK}$. This means that $(B^\circ,J,\omega_{\rm SK})$ is a K\"ahler manifold and there is a torsion-free flat connection $\nabla^{\rm SK}$ on $T^\mathbb{R}B^\circ$ which makes $\omega_{\rm SK}$ parallel and $d^{\nabla^{\rm SK}} J=0$ (however, in general $\nabla^{\rm SK}J\neq 0$), where $d^{\nabla^{\rm SK}}:\Omega^1(T^{\mathbb{R}}B^\circ)\to \Omega^2(T^{\mathbb{R}}B^\circ)$ is the usual extension of $\nabla^{\rm SK}$ (cf. \cite[p.33]{Fr}). The Riemannian metric associated to $\omega_{\rm SK}$ will be denoted by $g_{\rm SK}$ and its Levi-Civita/Chern connection, which in general is different from $\nabla^{\rm SK}$, will be denoted simply by $\nabla$ (see \eqref{freed} below for an explicit formula relating $\nabla$ and $\nabla^{\rm SK}$). On every sufficiently small open set $U\subset B^\circ$ we can find special holomorphic local coordinates $\{z_j\}_{j=1}^n$ (whose real parts are flat Darboux coordinates) and a holomorphic map $Z:U\to \mathfrak{H}_n$ into the Siegel upper half space
$$\mathfrak{H}_n=\{A\in\mathfrak{gl}(n,\mathbb{C})\ |\ A=A^t,\quad \mathrm{Im}A>0\},$$
such that $Z(y)$ are the periods of the torus fiber $f^{-1}(y)$, and we can write
$$\omega_{\rm SK}=\frac{1}{2}\sum_{i,j}\mathrm{Im}Z_{ij}dz_i\wedge d\ov{z}_j.$$
It is also worth noting that special K\"ahler manifolds can only be complete if they are flat, by a result of Lu \cite{Lu}. See \cite{TZ} for a description of the metric completion of $(B^\circ,\omega_{\rm SK})$ and of its metric singularities.

Special K\"ahler metrics have a Hodge-theoretic origin (see \cite{Her}, \cite{Ma2}): as mentioned earlier there is a natural weight-one polarized real variation of Hodge structures $R^1f_*\mathbb{R}_{X^\circ}$ on $B^\circ$, whose Hodge bundle of type $(1,0)$ is isomorphic to $TB^\circ$ (by contracting with the holomorphic symplectic form), and its Hodge metric is exactly the special K\"ahler metric. 

In \cite{To,GTZ,GTZ2} it is also shown that $\omega_{\rm SK}$ can be written as $\omega_B+\ddbar\vp$ for some K\"ahler metric $\omega_B$ on $B$ and some function $\vp\in C^\infty(B^\circ)\cap L^\infty(B)$. In fact, a priori there is a different special K\"ahler metric on $B^\circ$ for each chosen K\"ahler class  $[\omega_B]$ on $B$, but since $B$ is smooth Fano and of Picard number one, it follows that $b_2(B)=1$ so there is a unique choice of K\"ahler class up to scaling. In the following, we fix one such $\omega_B$ once and for all. This way, we can unambiguously talk about ``the'' special K\"ahler metric $\omega_{\rm SK}$ in the following.
\subsection{Curvature properties}
Following Freed \cite{Fr}, there is a holomorphic symmetric cubic form $\Xi\in H^0(B^\circ, \mathrm{Sym}^3 T^*B^\circ)$
such that, in any local holomorphic coordinate system, the curvature tensor of $\omega_{\rm SK}$ can be written as
\begin{equation}\label{curv}
R_{i\ov{j}k\ov{\ell}}=g_{\rm SK}^{p\ov{q}}\Xi_{ikp}\ov{\Xi_{j\ell q}},
\end{equation}
and on any sufficiently small $U$ as above we can find a holomorphic function $\mathcal{F}:U\to\mathbb{C}$ such that, in special holomorphic coordinates, the period matrix and the cubic form can be written as
\begin{equation}\label{deriv}
Z_{ij}=\frac{\de^2\mathcal{F}}{\de z_i\de z_j},\quad \Xi_{ipk}=\frac{\de^3\mathcal{F}}{\de z_i\de z_k\de z_p}.
\end{equation}
From the curvature formula \eqref{curv} we see in particular that $\omega_{\rm SK}$ has nonnegative bisectional curvature on $B^\circ$: given any $v,w\in T^{1,0}B^\circ$ we have
$$\mathrm{Rm}(v,\ov{v},w,\ov{w})=R_{i\ov{j}k\ov{\ell}}v^i\ov{v^j}w^k\ov{w^\ell}=\sum_{p}\left|\Xi(v,w,e_p)\right|^2\geq 0,$$
where $\{e_p\}$ is any $g_{\rm SK}$-unitary frame.

We will also use the following dichotomy, which was conjectured by Matsushita, and after progress by van Geemen-Voisin \cite{vV} it was recently proved by Bakker \cite{Ba}:
\begin{theorem}\label{prolisso2}
Either $f$ is isotrivial, or else $f$ has maximal variation.
\end{theorem}
This dichotomy is then reflected in the curvature properties of $\omega_{\rm SK}$:
\begin{corollary}\label{prolisso}
Either $\omega_{\rm SK}$ is flat on $B^\circ$, or else $\omega_{\rm SK}$ has positive Ricci curvature on a Zariski open subset of $B^\circ$.
\end{corollary}
In the second case, up to replacing $D$ with a larger closed analytic subvariety we will always assume that $\Ric_{g_{\rm SK}}>0$ on $B^\circ$.
\begin{proof}
We use Bakker's Theorem \ref{prolisso2}. If $f$ is isotrivial, then the local period map $Z$ is constant, so from \eqref{deriv} we see that $\Xi\equiv 0$ on $B^\circ$, and \eqref{curv} shows that $\omega_{\rm SK}$ is flat. If $f$ has maximal variation, then the period map $Z$ is generically of maximal rank (equal to $n$), so $Z$ is an immersion on a Zariski open subset of $B^\circ$ (which, up to enlarging $D$, we may assume is equal to $B^\circ$). Given  any $v\in T^{1,0}_x B^\circ$, the Ricci curvature of $\omega_{\rm SK}$ in the direction of $v$ is given by
$$\Ric(v,\ov{v})=\sum_{p,q}\left|\Xi(v,e_p,e_q)\right|^2\geq 0,$$
and if this vanishes for some $v\neq 0$ then in special holomorphic coordinates we have that for all $p,q$
$$0=\Xi(v,e_p,e_q)=\frac{\de^3\mathcal{F}}{\de v\de e_p\de e_q}=\frac{\de}{\de v}Z_{pq},$$
so the period map is not an immersion at $x$, a contradiction.
\end{proof}

\begin{rmk}
The holomorphic sectional curvature of $\omega_{\rm SK}$ is given by
$${\rm HSC}(v)=R_{i\ov{j}k\ov{\ell}}v^i\ov{v^j}v^k\ov{v^\ell}=\sum_{p}\left|\frac{\de^3\mathcal{F}}{\de v\de v\de e_p}\right|^2\geq 0,$$
where $v\in T^{1,0}B^\circ$ is a unit vector (and in the last equality we use special holomorphic coordinates). The condition that $\omega_{\rm SK}$ has (strictly) positive holomorphic sectional curvature on $B^\circ$ thus means that none of the ``diagonal'' entries of the period matrix $Z$
$$Z_{ij}v^iv^j=\frac{\de^2\mathcal{F}}{\de v\de v}$$
is locally constant. We expect that this always holds (up to enlarging $D$) when $f$ has maximal variation.
\end{rmk}

\begin{rmk}
We are grateful to B. Bakker for the following observation.
Let $f:S\to \mathbb{P}^1$ be an elliptic fibration of a $K3$ surface $S$. For $n\geq 2$ let $X=S^{[n]}$ be the Hilbert scheme parametrizing length $n$ subschemes of $S$.
We obtain an induced holomorphic Lagrangian fibration $\ti{f}:X\to (\mathbb{P}^1)^{[n]}=\mathbb{P}^n$ whose general fiber is isomorphic to the product of $n$ general fibers of $f$, and if $f$ has maximal variation then so does $\ti{f}$. Since the period matrix of such a torus is diagonal, we see that the period map $Z$ of $\ti{f}$ has $Z_{ij}=0$ for $i\neq j$. It follows that for these examples the special K\"ahler metric, which is not flat if $f$ has maximal variation, nevertheless does not have strictly positive bisectional curvature on $\mathbb{P}^n\backslash D$, since in local special coordinates we have
$$\mathrm{Rm}(e_i,\ov{e_i},e_j,\ov{e_j})=R_{i\ov{i}j\ov{j}}=0,$$
for all $i\neq j$.
Thus, to prove our main theorem, it would not be sufficient to prove a suitable noncompact version of the Mori-Siu-Yau theorem \cite{Mor, SY2}, but we must instead generalize the work of Mok \cite{Mok}.
\end{rmk}

\section{Estimates on the special K\"ahler metric}\label{se3}
We collect in this section two crucial estimates for the special K\"ahler metric $\omega_{\rm SK}$, which are contained or follow from earlier work of the second-named author and coauthors \cite{To,TZ3,TZ}. See also \cite{CH,Ha} for a study of the asymptotics of special K\"ahler metrics on Riemann surfaces.
\subsection{Strict positivity} The first estimate, taken from \cite{GTZ,GTZ2,To,TZ3}, says that the positivity of $\omega_{\rm SK}$ does not degenerate as we approach $D$. Since this statement is valid even if $B$ is singular, we present it in this generality.
\begin{proposition}
Let $X$ be a hyperk\"ahler manifold, $f:X\to B$ a holomorphic Lagrangian fibration with $B$ a normal analytic variety. Let $\omega_B$ be a smooth K\"ahler metric on $B$ (in the sense of analytic spaces) and $\omega_{\rm SK}$ the special K\"ahler metric on $B^\circ$ cohomologous to $\omega_B$. Then there is $C>0$ such that on $B^\circ$ we have
\begin{equation}\label{sl}
\omega_{\rm SK}\geq C^{-1}\omega_B.
\end{equation}
\end{proposition}
\begin{proof}
Fix a K\"ahler metric $\omega_X$ on $X$ and for $t\geq 0$ let $\omega_t$ be the hyperk\"ahler metric on $X$ cohomologous to $f^*\omega_B+e^{-t}\omega_X$. Then the Schwarz Lemma \cite[Lemma 3.1]{To} (using also \cite[Proof of Theorem 3.2]{TZ3} in the case when $B$ is singular) gives $$\omega_t\geq C^{-1}f^*\omega_B,$$
on $X^\circ$ (with $C$ independent of $t\geq 0$), and thanks to \cite[Theorem 1.1]{GTZ}, \cite{HT} and \cite[Theorem 1.2]{GTZ2} we know that as $t\to\infty$ we have
$$\omega_t\to f^*\omega_{\rm SK},$$
locally uniformly on $X^\circ$ (and even locally smoothly), so we conclude that
$$f^*\omega_{\rm SK}\geq C^{-1}f^*\omega_B,$$
on $X^\circ$, and since $f$ is a submersion over $B^\circ$ this is equivalent to
$$\omega_{\rm SK}\geq C^{-1}\omega_B,$$
on $B^\circ$.
\end{proof}

\begin{rmk}
If $B$ has quotient singularities (which is expected to hold in general \cite[Remark 1.11]{HM}) then we can replace $\omega_B$ with an orbifold K\"ahler metric $\omega_{\rm orb}$, and a similar argument gives the stronger bound
$$\omega_{\rm SK}\geq C^{-1}\omega_{\rm orb},$$
on $B^\circ$.
\end{rmk}

\subsection{Ricci curvature bounds near $D$}\label{as}
From now on, we return to our standing assumption that $B$ is smooth. The second crucial estimate is a bound for the Ricci curvature of $\omega_{\rm SK}$. We have seen in the previous section that $\omega_{\rm SK}$ has nonnegative Ricci curvature on $B^\circ$. In fact, as shown in \cite{ST, To} (see also \cite[Prop. 4.1]{TZ}), we have
$$\Ric_{g_{\rm SK}}=\omega_{\rm WP}\geq 0,$$
where $\omega_{\rm WP}$ is the Weil-Petersson form of the family of abelian varieties $f:X^\circ\to B^\circ$ (pullback of the Weil-Petersson metric on the moduli space via the moduli map). Concretely, on $B^\circ$ we have
\begin{equation}\label{fibber}
\omega_{\rm SK}^n=c (-1)^{\frac{n^2}{2}}f_*(\sigma^n\wedge\ov{\sigma^n}),
\end{equation}
where $c>0$ and $\sigma$ is a holomorphic symplectic form on $X$, and to obtain $\omega_{\rm WP}$ it suffices to take $-\ddbar\log$ of the fiber integral in \eqref{fibber} divided by the local Euclidean volume form.

Recall that the discriminant locus $D\subset B$ is a closed analytic subvariety of pure codimension $1$, see  \cite[Proposition 3.1]{HO}. Let $x\in D$ be any smooth point of $D$, and choose an open neighborhood $U$ of $x$ with local  holomorphic coordinates centered at $x$ such that $D\cap U=\{z_1=0\}$. Thus, at points of $D\cap U$, the vectors $\frac{\de}{\de z_2},\dots,\frac{\de}{\de z_n}$ are tangent to $D$, while $\frac{\de}{\de z_1}$ is transversal. The main claim is the following:
\begin{proposition}\label{riccicurb} On $\{z_1\neq 0\}$ the Ricci curvature tensor $R_{i\ov{j}}=\Ric_{g_{\rm SK}}\left(\frac{\de}{\de z_i},\frac{\de}{\de \ov{z_j}}\right)$ of $\omega_{\rm SK}$ satisfies
\begin{equation}\label{ricci1}
0\leq R_{i\ov{i}}\leq C,\quad 2\leq i\leq n,
\end{equation}
\begin{equation}\label{ricci2}
0\leq R_{1\ov{1}}\leq \frac{C}{|z_1|^2},
\end{equation}
for some constant $C>0$.
\end{proposition}

\begin{proof}
We will use freely the arguments in \cite[\S 4.3]{TZ} (these are stated for $X$ projective hyperk\"ahler, but all arguments there go through for general $X$ hyperk\"ahler using that Lagrangian fibrations are locally projective \cite{Cam}).
By the Monodromy Theorem, there is $m\in\mathbb{N}_{>0}$ such that the eigenvalues of the monodromy operator $T$ (acting on $H^1(f^{-1}(y),\mathbb{Z})$ for some fixed basepoint $y\in U\backslash D$) are $m$th roots of unity. We may assume without loss that in our coordinates $U$ is the unit polydisc, and letting $\ti{U}$ be the unit polydisc with coordinates $(t_1,\dots,t_n)$, we define the branched covering
$$q:\ti{U}\to U, \quad q(t_1,\dots, t_n)=(t_1^m,t_2,\dots,t_n).$$
Then after pulling back to $\ti{U}$, the monodromy operator $T$ becomes unipotent, with
$$(T-\mathrm{Id})^2=0.$$
Thanks to the argument in \cite[P. 774]{TZ}, we can find holomorphic functions $w_1,\dots, w_n$ on $\ti{U}$, which are special holomorphic coordinates on $\ti{U}\cap\{t_1\neq 0\}$ (but need not form a coordinate system at points on $\{t_1=0\}$, and they may even vanish there), such that, on $\ti{U}\cap\{t_1\neq 0\}$, we can write
$$q^*\omega_{\rm SK}=\frac{i}{2}\sum_{j,k}\mathrm{Im}Z_{jk}(t) dw_j \wedge d\ov{w_k}=\frac{i}{2}\sum_{j,k,p,q}\mathrm{Im}Z_{jk}(t)\frac{\de w_j}{\de t_p}\ov{\frac{\de w_k}{\de t_q}} dt_p \wedge d\ov{t_q},$$
where $Z_{jk}(t)$ is the local period map pulled back to $\ti{U}$. Thus, if we denote by $dV_E$ the Euclidean volume form on $\ti{U}$ given by the coordinates $t_1,\dots t_n$, we have
$$\log\frac{q^*\omega^n_{\rm SK}}{dV_E}=\log\det \mathrm{Im}Z + \log \left|\det\left(\frac{\de w_j}{\de t_p}\right)\right|^2,$$
and since $\det\left(\frac{\de w_j}{\de t_p}\right)$ is holomorphic and nonzero on $\ti{U}\cap\{t_1\neq 0\}$, we get
\begin{equation}\label{sav4}
\Ric_{q^*g_{\rm SK}}\left(\frac{\de}{\de t_j},\frac{\de}{\de \ov{t_k}}\right)=-\frac{\de}{\de t_j}\frac{\de}{\de \ov{t_k}}\log\frac{q^*\omega^n_{\rm SK}}{dV_E}=-\frac{\de}{\de t_j}\frac{\de}{\de \ov{t_k}}\log\det \mathrm{Im}Z.
\end{equation}
To estimate this, following  \cite[Lemma 4.3]{TZ} we use Schmid's Nilpotent Orbit Theorem \cite{Sch} and see there are $b_{jk}\in\mathbb{Q}$ and a holomorphic map
$Q$ from $\ti{U}$ to the space of symmetric $n\times n$ complex matrices, such that on $\ti{U}\cap\{t_1\neq 0\}$ we have
$$Z_{jk}(t)=Q_{jk}(t)+\frac{\log t_1}{2\pi i}b_{jk},\quad 1\leq j,k\leq n,$$
for some branch of $\log$. Thus,
\begin{equation}\label{sav1}
\mathrm{Im} Z_{jk}(t)=\mathrm{Im} Q_{jk}(t)-\frac{b_{jk}}{2\pi}\log|t_1|,
\end{equation}
and furthermore (see \cite[Lemma 4.3]{TZ}) there is $C>0$ such that on $\ti{U}\cap\{t_1\neq 0\}$ we have
\begin{equation}\label{sav2}
\mathrm{Im}Z(t) \geq C^{-1}\mathrm{Id},
\end{equation}
and so the inverse matrix of $\mathrm{Im}Z(t)$, whose entries will be denoted by $(\mathrm{Im}Z(t))^{pq}$, satisfies
\begin{equation}\label{sav3}
0<(\mathrm{Im}Z(t))^{-1} \leq C\mathrm{Id}.
\end{equation}
Differentiating the determinant gives
\[\begin{split}
&-\frac{\de}{\de t_j}\frac{\de}{\de \ov{t_k}}\log\det \mathrm{Im}Z(t)=-(\mathrm{Im}Z(t))^{pq}\frac{\de}{\de t_j}\frac{\de}{\de \ov{t_k}}\mathrm{Im}Z_{pq}(t)\\
&+(\mathrm{Im}Z(t))^{pq}(\mathrm{Im}Z(t))^{rs}\frac{\de}{\de t_j}\mathrm{Im}Z_{pr}(t)\frac{\de}{\de \ov{t_k}}\mathrm{Im}Z_{qs}(t).
\end{split}\]
First we take $j\geq 2$, and differentiating \eqref{sav1} gives
\[\begin{split}
&-\frac{\de}{\de t_j}\frac{\de}{\de \ov{t_j}}\log\det \mathrm{Im}Z(t)=(\mathrm{Im}Z(t))^{pq}(\mathrm{Im}Z(t))^{rs}\frac{\de}{\de t_j}\mathrm{Im}Q_{pr}(t)\frac{\de}{\de \ov{t_j}}\mathrm{Im}Q_{qs}(t)
\leq C,
\end{split}\]
using \eqref{sav3} and the fact that $Q$ is holomorphic on all of $\ti{U}$.
As for the $t_1$ direction, differentiating \eqref{sav1} we have
\[\begin{split}
&-\frac{\de}{\de t_1}\frac{\de}{\de \ov{t_1}}\log\det \mathrm{Im}Z(t)\\
&=(\mathrm{Im}Z(t))^{pq}(\mathrm{Im}Z(t))^{rs}\frac{\de}{\de t_1}\left(\mathrm{Im}Q_{pr}(t)-\frac{b_{pr}}{2\pi}\log|t_1|\right)
\frac{\de}{\de \ov{t_1}}\left(\mathrm{Im}Q_{qs}(t)-\frac{b_{qs}}{2\pi}\log|t_1|\right)\\
&\leq C\sum_{p,r}\left|\frac{\de}{\de t_1}\left(\mathrm{Im}Q_{pr}(t)-\frac{b_{pr}}{2\pi}\log|t_1|\right)\right|^2\\
&\leq C+C\left|\frac{\de}{\de t_1}\log|t_1|\right|^2\\
&\leq \frac{C}{|t_1|^2}.
\end{split}\]
Going back to \eqref{sav4}, this shows that on $\ti{U}\cap\{t_1\neq 0\}$) we have
$$0\leq \Ric_{q^*g_{\rm SK}}\left(\frac{\de}{\de t_j},\frac{\de}{\de \ov{t_j}}\right)\leq C,\quad j\geq 2,$$
$$0\leq\Ric_{q^*g_{\rm SK}}\left(\frac{\de}{\de t_1},\frac{\de}{\de \ov{t_1}}\right)\leq  \frac{C}{|t_1|^2},$$
and so on $U\cap\{z_1\neq 0\}$ we have for $j\geq 2$,
$$0\leq R_{j\ov{j}}=\Ric_{g_{\rm SK}}\left(\frac{\de}{\de z_j},\frac{\de}{\de \ov{z_j}}\right)=\Ric_{q^*g_{\rm SK}}\left(\frac{\de}{\de t_j},\frac{\de}{\de \ov{t_j}}\right)\leq C,$$
and
\[\begin{split}0\leq R_{1\ov{1}}&=\Ric_{g_{\rm SK}}\left(\frac{\de}{\de z_1},\frac{\de}{\de \ov{z_1}}\right)=\frac{1}{m^2|t_1|^{2m-2}}\Ric_{q^*g_{\rm SK}}\left(\frac{\de}{\de t_1},\frac{\de}{\de \ov{t_1}}\right)\\
&\leq \frac{C}{m^2|t_1|^{2m}}\leq\frac{C}{|z_1|^2},
\end{split}\]
as desired.
\end{proof}

\begin{rmk}
We expect that the sharp bound in \eqref{ricci2} in general is of the form $\frac{C}{|z_1|^2\log^2|z_1|}$, cf. \cite{Yosh} when $\dim B=1$. One may be able to show this by proving an asymptotic expansion for the fiber integral in \eqref{fibber} which can be differentiated term-by-term, as in \cite{Bar,Taka2}.
\end{rmk}

\section{Rational curves and rigidity}\label{se4}
Recall that $B$ is a Fano manifold, hence uniruled. Let $\nu:\mathbb{P}^1\to B$ be a rational curve (i.e. a nonconstant holomorphic map) whose image is not contained in $D$. Our first result of this section shows that $\nu$ is a {\em free} rational curve, in the terminology of Mori Theory, cf. \cite{Kol2}.

\subsection{Freeness of the rational curve}
By Grothendieck's Theorem, the vector bundle $\nu^* TB$ splits and so we can write
\begin{equation}\label{h1}
\nu^*TB\cong\bigoplus_{i=1}^n\mathcal{O}(a_i),
\end{equation}
for some integers $a_i$, which we order by $a_1\geq\cdots\geq a_n.$
Dualizing, we have
\begin{equation}\label{h6}
\nu^*\Omega^1_{B}\cong\bigoplus_{i=1}^n\mathcal{O}(-a_i),
\end{equation}
and
\begin{equation}\label{deg}
q:=-K_B\cdot\nu(\mathbb{P}^1)=\sum_{i=1}^n a_i>0,
\end{equation}
since $B$ is Fano.

On $B^\circ$ we equip $\Omega^{1}_B$ with the Hermitian metric $h_{\rm SK}$ induced by the special K\"ahler metric $\omega_{\rm SK}$.

\begin{lemma}\label{nonn} We have $a_n\geq 0$.
\end{lemma}
\begin{proof}
This argument was suggested to us by M. P\u{a}un. Consider the nontrivial section $v\in H^0(\mathbb{P}^1,\nu^*\Omega^1_B\otimes \mathcal{O}(a_n))$ which corresponds to the quotient morphism $\nu^*TB\to \mathcal{O}(a_n)$.
Equip $L:=\mathcal{O}(a_n)$ with a smooth metric $h_L$ on $\mathbb{P}^1$, and equip $\nu^*\Omega^1_B$ with the smooth metric $\nu^*h_{\rm SK}$ on $\mathbb{P}^1\backslash \nu^{-1}(D)$ which is the pullback of the metric induced by $\omega_{\rm SK}$. Thus, the curvature of $\nu^*h_{\rm SK}$ is Griffiths nonpositive on $\mathbb{P}^1\backslash \nu^{-1}(D)$, since $\omega_{\rm SK}$ has nonnegative bisectional curvature on $B^\circ$ and dualization reverses the sign of Griffiths positivity (see e.g. \cite[\S VII.6]{De}). Equip then $\nu^*\Omega^1_B\otimes \mathcal{O}(a_n)$ with the metric $h=\nu^*h_{\rm SK}\otimes h_L$ on $\mathbb{P}^1\backslash \nu^{-1}(D)$.

Differentiating $\log|v|^2_h$ on $\mathbb{P}^1\backslash \nu^{-1}(D)$ we have the well-known identity of $(1,1)$-forms on $\mathbb{P}^1\backslash \nu^{-1}(D)$
$$\ddbar\log|v|^2_h=\frac{|\nabla v|^2_h}{|v|^2_h}-\frac{|\langle\nabla v,v\rangle_h|^2}{|v|^4_h}-R_{h_L}-\frac{\langle R_{\nu^*h_{\rm SK}}(v),v\rangle_h}{|v|^2_h},$$
where $\nabla v$ is an $\nu^*\Omega^1_B\otimes \mathcal{O}(a_n)$-valued $(1,0)$-form, so $|\nabla v|^2_h$ is a $(1,1)$-form, and similarly for the other terms. Using Cauchy-Schwarz we have
$$\frac{|\langle\nabla v,v\rangle_h|^2}{|v|^4_h}\leq\frac{|\nabla v|^2_h}{|v|^2_h},$$
and since on $\mathbb{P}^1\backslash \nu^{-1}(D)$ the curvature of $\nu^*h_{\rm SK}$ is Griffiths nonpositive, we can estimate
\begin{equation}\label{xyp}\begin{split}
\ddbar\log|v|^2_h&\geq -R_{h_L}-\frac{\langle R_{\nu^*h_{\rm SK}}(v),v\rangle_h}{|v|^2_h}\\
&\geq -R_{h_L},
\end{split}\end{equation}
 Since $R_{h_L}$ is a smooth form on $\mathbb{P}^1$, we see that $\log|v|^2_h$ is quasi-psh on $\mathbb{P}^1\backslash \nu^{-1}(D)$, and using \eqref{sl} we see that
$$\sup_{\mathbb{P}^1\backslash \nu^{-1}(D)}\log|v|^2_h\leq C+\sup_{\mathbb{P}^1\backslash \nu^{-1}(D)}\log|v|^2_{\nu^*h_B\otimes h_L}<\infty,$$
where $h_B$ is the smooth metric on $\Omega^1_B$ induced by $\omega_B$. Thus $\log|v|^2_h$ is bounded above, hence by the Grauert-Remmert extension theorem \cite{GR} the inequality $R_{h_L}+\ddbar\log|v|^2_h\geq 0$ extends over the singularities to all of $\mathbb{P}^1$ (in the weak sense). Integrating this over $\mathbb{P}^1$ and using Stokes thus gives
$$a_n=\int_{\mathbb{P}^1}R_{h_L}=\int_{\mathbb{P}^1}(R_{h_L}+\ddbar\log|v|^2_h)\geq 0,$$
as desired.
\end{proof}
Lemma \ref{nonn} says that every rational curve in $B$ which is not contained in $D$ is free, and by Mori Theory it deforms to cover a Zariski dense subset of $B$ (see e.g. \cite{Kol2}).

The pullback morphism $\nu^*\Omega^1_{B}\to \Omega^1_{\mathbb{P}^1}$ dualizes to a nontrivial morphism $\mathcal{O}(2)\to \nu^*TB$, and hence $a_1\geq 2$. Using this observation and Lemma \ref{nonn} we can write the splittings in \eqref{h1} and \eqref{h6} as
\begin{equation}\label{h5}
\nu^*TB\cong\mathcal{O}(a_1)\oplus\cdots\oplus \mathcal{O}(a_{n-\ell})\oplus \mathcal{O}^{\oplus\ell},
\end{equation}
\begin{equation}\label{h7}
\nu^*\Omega^1_{B}\cong\mathcal{O}(-a_1)\oplus\cdots\oplus \mathcal{O}(-a_{n-\ell})\oplus \mathcal{O}^{\oplus\ell},
\end{equation}
for some $0\leq \ell\leq n-1$, where $a_1\geq a_2\geq\cdots\geq a_{n-\ell}\geq 1, a_1\geq 2,$ and $$q=\sum_{i=1}^{n-\ell}a_i.$$

Recall now a result by Cho-Miyaoka-Shepherd-Barron \cite[Cor. 0.4(11)]{CMSB}, which uses Mori theory:
\begin{theorem}\label{unir}
Let $B$ be a uniruled projective manifold, $D$ an effective divisor, and suppose that, for any rational curve $\nu:\mathbb{P}^1\to B$ which is not contained in $D$, we have the inequality
\begin{equation}\label{agognata}
-K_{B}\cdot \nu(\mathbb{P}^1)\geq n+1.
\end{equation}
Then $B\cong\mathbb{P}^n$.
\end{theorem}

If, in our setting, for all rational curves $\nu:\mathbb{P}^1\to B$ not contained in $D$ we have $\ell=0$, i.e. $a_i>0$ for all $i$, then since $a_1\geq 2$ it would follow that $-K_{B}\cdot \nu(\mathbb{P}^1)=\sum_{i=1}^na_i\geq n+1$ and so $B$ would be isomorphic to $\mathbb{P}^n$. In other words, if $B\not\cong\mathbb{P}^n$ then there exists a rational curve $\nu_0:\mathbb{P}^1\to B$ not contained in $D$ which has $\ell\geq 1$, i.e. there are some trivial factors $\mathcal{O}^{\oplus\ell}$ in the splitting \eqref{h1}. We may also assume that the anticanonical degree $q:=-K_{B}\cdot \nu_0(\mathbb{P}^1)$ is as small as possible among all rational curves not contained in $D$ (and satisfies $2\leq q\leq n$), and we will call these {\em minimal degree rational curves}, which is consistent with the standard terminology, e.g. in \cite{Hw3}. By Lemma \ref{nonn}, this rational curve $\nu_0$ is free and so it deforms to cover a Zariski dense subset of $B$. Let $\mathcal{K}$ be the irreducible component of the space of rational curves in $B$ (see \cite[\S II.2]{Kol2}) which contains $\nu_0$, which we fix once and for all. From Mori Theory (see \cite{Kol2} and \cite[\S 3]{Hw3}) we have that $\mathcal{K}$ is a quasiprojective variety equipped with a universal $\mathbb{P}^1$-bundle $\rho:\mathcal{U}\to\mathcal{K}$ and an evaluation map $\mu:\mathcal{U}\to B$. For any $t\in\mathcal{K}$ we will also write
$$\mathcal{U}_t:=\rho^{-1}(t)\subset\mathcal{U},$$
so $\mathcal{U}_t\cong\mathbb{P}^1$ is the rational curve corresponding to $t$, and
$$\nu_t:=\mu|_{\mathcal{U}_t}:\mathcal{U}_t\to B,$$
will denote the morphism to $B$.

Furthermore, the generic rational curve in $\mathcal{K}$ is free and not contained in $D$,  $\mathcal{K}$ is smooth at such curves, and the integers $a_i,\ell$ in the decomposition \eqref{h5} are the same for all generic such curves. Given $x\in B^\circ$ there is some minimal degree rational curve $\nu$ in $\mathcal{K}$ that passes through $x$ and is smooth at $x$. Thanks to \cite[Prop. II.3.7]{Kol2}, we can also assume that $\nu(\mathbb{P}^1)$ intersects $D$ only at the regular points of $D$ (since the singularities of $D$ have codimension at least $2$ in $B$), and that these intersections are transverse.
The evaluation morphism $\mu:\mathcal{U}\to B$ is a submersion over a Zariski open subset of $B$, which up to enlarging $D$ we may assume equals $B^\circ$. Thus, if we define
$\mathcal{U}^\circ:=\mu^{-1}(B^\circ),$ then $\mathcal{U}^\circ$ is smooth and $\mu:\mathcal{U}^\circ\to B^\circ$ is a submersion.
The metric $g_{\rm SK}$ on $TB^\circ$ induces by pullback a metric $\mu^*TB^\circ$ over $\mathcal{U}^\circ$, which we will denote by the same symbol, and similarly for the connections $\nabla$ and $\nabla^{\rm SK}$, which induce pullback connections denoted in the same way.

\begin{lemma}
There is a locally free sheaf $\mathcal{V}^\sharp$ on $\mathcal{U}$ such that for every $t\in\mathcal{K}$, the restriction $\mathcal{V}^\sharp\big|_{\mathcal{U}_t}$ of $\mathcal{V}^\sharp$ to the rational curve $\mathcal{U}_t$ equals the factor $\mathcal{O}^{\oplus\ell}$  in the splitting \eqref{h7} for $\nu_t^*\Omega^1_B$.
\end{lemma}
\begin{proof}
For the sake of clarity, we first define the fiber $\mathcal{V}^\sharp$ at any point on $\mathcal{U}_t\cong\mathbb{P}^1$. For this, we consider $\nu_t^*\Omega^1_B,$ which from the splitting \eqref{h7} is isomorphic to
$\mathcal{O}(-a_1)\oplus\cdots\oplus \mathcal{O}(-a_{n-\ell})\oplus \mathcal{O}^{\oplus\ell}$. Its space of global sections $H^0(\mathcal{U}_t, \nu_t^*\Omega^1_B)$ is then $\ell$-dimensional, and we can find a basis of such sections which are linearly independent at all points of $\mathbb{P}^1$. The fiber of $\mathcal{V}^\sharp$ at any point on $\mathcal{U}_t$ is then defined as the linear span of any given basis of $H^0(\mathcal{U}_t, \nu_t^*\Omega^1_B)$.

To prove that this collection of $\ell$-dimensional vector spaces form a locally free sheaf, consider first the locally free sheaf $\mu^*\Omega^1_B$ on $\mathcal{U}$, and take its
direct image sheaf $\rho_*\mu^*\Omega^1_B$.
Since $h^0(\mathcal{U}_t, \mu^*\Omega^1_B|_{\mathcal{U}_t})=\ell$ is independent of $t$, Grauert's Theorem on direct images \cite[Cor. III.12.9]{Har} shows that $\rho_*\mu^*\Omega^1_B$ is a locally free sheaf on $\mathcal{K}$. We then set $\mathcal{V}^\sharp=\rho^*\rho_*\mu^*\Omega^1_B$, which is a locally free sheaf over $\mathcal{U}$ whose fibers agree with our previous description.
\end{proof}

Our main interest will be with the restriction of $\mathcal{V}^\sharp$ to $\mathcal{U}^\circ$, which will be denoted with the same notation. This is a holomorphic vector bundle over $\mathcal{U}^\circ$, which is naturally a subbundle of $\mu^*\Omega^1_{B^\circ}$.  We then define a holomorphic subbundle  $\mathcal{V}\subset \mu^*TB^\circ$ over $\mathcal{U}^\circ$ as the annihilator of $\mathcal{V}^\sharp$, namely
$$\mathcal{V}=\{v\in \mu^*TB^\circ\ |\ \gamma(v)=0,\ \text{for all }\gamma\in \mathcal{V}^\sharp\}.$$
For any $t\in\mathcal{K}$ we have that the restriction of $\mathcal{V}$ to $\mathcal{U}_t$ equals the factor $\mathcal{O}(a_1)\oplus\cdots\oplus \mathcal{O}(a_{n-\ell})$ in the splitting \eqref{h5} for $\nu_t^*TB$. Observe that since the pullback morphism $\nu_t^*\Omega^1_{B}\to \Omega^1_{\mathbb{P}^1}$ dualizes to a nontrivial morphism $\mathcal{O}(2)\to \nu_t^*TB$, it follows that the tangent direction to the image of $\nu_t$ at any point on this curve (which is a line in $TB^\circ$) when pulled back via $\mu$ lies in the fiber of $\mathcal{V}$ over $\mathcal{U}_t$.

We then define a smooth complex subbundle $\mathcal{N}\subset \mu^*TB^\circ$ over $\mathcal{U}^\circ$ as the $g_{\rm SK}$-orthogonal complement of $\mathcal{V}$, and $\mathcal{N}^\sharp\subset \mu^*\Omega^1_{B^\circ}$ as its annihilator (or equivalently as the  $g_{\rm SK}$-orthogonal complement of $\mathcal{V}^\sharp$), so that over $\mathcal{U}^\circ$ we have the splittings
\begin{equation}\label{ancorasplit}
\mu^*TB^\circ=\mathcal{V}\oplus\mathcal{N}, \quad \mu^*\Omega^1_{B^\circ}=\mathcal{V}^\sharp\oplus\mathcal{N}^\sharp.
\end{equation}
The bundles $\mathcal{N},\mathcal{N}^\sharp$ are not yet known to be holomorphic (we will prove this later on).
Note also that the (complex antilinear) smooth isomorphism
\begin{equation}\label{isow}
\mu^*TB^\circ\to\mu^*\Omega^1_{B^\circ},
\end{equation}
defined by the metric $g_{\rm SK}$ (by ``lowering the index'' and conjugating) maps $\mathcal{N}$ isomorphically onto $\mathcal{V}^\sharp$.

\subsection{The rigidity theorem}We have the following rigidity statement:
\begin{theorem}\label{rigid}
Given a rational curve $\mathcal{U}_t$ for some $t\in \mathcal{K}$, with morphism $\nu_t:\mathbb{P}^1\to B$, and given a section $u\in H^0(\mathbb{P}^1,\mathcal{V}^\sharp\big|_{\mathcal{U}_t})$, let $\nu_t^*h_{\rm SK}$ be the smooth metric on $\nu_t^*\Omega^1_B$ over $\mathbb{P}^1\backslash\nu_t^{-1}(D)$ induced by $g_{\rm SK}$, and let $R_{\nu_t^*h_{\rm SK}}$ be its curvature. Then we have:
\begin{itemize}
\item[(a)] On $\mathbb{P}^1\backslash\nu_t^{-1}(D)$ we have
\begin{equation}\label{new1}
\langle R_{\nu_t^*h_{\rm SK}}(u),u\rangle_{\nu_t^*h_{\rm SK}}=0.
\end{equation}
\item[(b)] Let $\zeta$ be the smooth section of $\mathcal{N}\big|_{\mathcal{U}_t}$ over $\mathbb{P}^1\backslash\nu_t^{-1}(D)$ which corresponds to $u$ under \eqref{isow}, and let $\alpha$ be a tangent vector to $\nu_t(\mathbb{P}^1)$. Then at any point on $\nu_t(\mathbb{P}^1)\cap B^\circ$ the curvature tensor of $g_{\rm SK}$ satisfies
\begin{equation}\label{crux}
R_{\alpha\ov{\alpha}\zeta\ov{\zeta}}=0,
\end{equation}
and hence
\begin{equation}\label{crux2}
\Xi(\alpha,\zeta,\beta)=0,\quad\text{for all }\beta\in TB^\circ.
\end{equation}
\item[(c)] For $\zeta$ as in (b), and for any section $v\in H^0(\mathbb{P}^1,\mathcal{V}\big|_{\mathcal{U}_t})$, at any point on $\nu_t(\mathbb{P}^1)\cap B^\circ$ we have
\begin{equation}\label{cruxa}
R_{v\ov{v}\zeta\ov{\zeta}}=0,
\end{equation}
as well as
\begin{equation}\label{cruxa2}
\Xi(v,\zeta,\beta)=0,\quad\text{for all }\beta\in TB^\circ.
\end{equation}
\item[(d)] Every section $u\in H^0(\mathbb{P}^1,\mathcal{V}^\sharp\big|_{\mathcal{U}_t})$ is parallel on $\mathbb{P}^1\backslash\nu_t^{-1}(D)$ with respect to the Chern connection $\nabla$ induced by $\omega_{\rm SK}$.
\item[(e)] The splitting $\nu_t^*\Omega^1_B=\mathcal{V}^\sharp\big|_{\mathcal{U}_t}\oplus \mathcal{N}^\sharp\big|_{\mathcal{U}_t}$ is preserved by $\nabla$.
\end{itemize}
\end{theorem}
\begin{proof}
(a) Equip $\mathcal{V}^\sharp\big|_{\mathcal{U}_t}$ with the smooth metric $h$ on $\mathbb{P}^1\backslash\nu_t^{-1}(D)$ induced by $\omega_{\rm SK}$ via $\mathcal{V}^\sharp\big|_{\mathcal{U}_t}\hookrightarrow\nu_t^*\Omega^1_B\to \Omega^1_B$.
Since $\omega_{\rm SK}$ has nonnegative bisectional curvature, the induced metric on $\Omega^1_B$ (and hence also the one on $\nu_t^*\Omega^1_B$) is Griffiths nonpositively curved, and since curvature decreases in subbundles, the metric $h$ is also Griffiths nonpositively curved.

As in \eqref{xyp}, on $\mathbb{P}^1\backslash \nu_t^{-1}(D)$ we have
\begin{equation}\label{xjp}\begin{split}
\ddbar\log|u|^2_h&=\frac{|\nabla u|^2_h}{|u|^2_h}-\frac{|\langle\nabla u,u\rangle_h|^2}{|u|^4_h}-\frac{\langle R_{h}(u),u\rangle_h}{|u|^2_h}\\
&\geq -\frac{\langle R_{h}(u),u\rangle_h}{|u|^2_h}\\
&\geq 0.
\end{split}\end{equation}
Thus $\log|u|^2_h$ is psh on $\mathbb{P}^1\backslash \nu_t^{-1}(D)$, and again using \eqref{sl} we see that
$$\sup_{\mathbb{P}^1\backslash \nu_t^{-1}(D)}\log|u|^2_h\leq C+\sup_{\mathbb{P}^1\backslash \nu_t^{-1}(D)}\log|u|^2_{\nu_t^*h_B}<\infty,$$
where $h_B$ is the smooth metric on $\Omega^1_B$ induced by $\omega_B$. Thus $\log|u|^2_h$ is bounded above, and by the Grauert-Remmert extension theorem \cite{GR} it extends to a global psh function on $\mathbb{P}^1$, which is therefore constant.

Thus $|u|^2_h$ is a nonzero constant, and from \eqref{xjp} we deduce that \begin{equation}\langle R_{h}(u),u\rangle_h=0,\end{equation} on $\mathbb{P}^1\backslash \nu_t^{-1}(D)$.
But using again the curvature decreasing property, we have
$$0=\langle R_{h}(u),u\rangle_h\leq \langle R_{\nu_t^*h_{\rm SK}}(u),u\rangle_{\nu_t^*h_{\rm SK}}\leq 0,$$
and so
\begin{equation}\langle R_{\nu_t^*h_{\rm SK}}(u),u\rangle_{\nu_t^*h_{\rm SK}}=0,\end{equation}
 on $\mathbb{P}^1\backslash \nu_t^{-1}(D)$, which proves \eqref{new1}.

(b) Since $\alpha\in TB^\circ$ is a tangent vector to $\nu_t(\mathbb{P}^1)$ and since $u$ is equal to the image of $\zeta$ under \eqref{isow}, we have
$$0=\langle R_{\nu_t^*h_{\rm SK}}(u),u\rangle_{\nu_t^*h_{\rm SK}}=-R_{\alpha\ov{\alpha}\zeta\ov{\zeta}},$$
which proves \eqref{crux}. The identity \eqref{crux2} is then a consequence of \eqref{curv}.

(c) Given $v\in H^0(\mathbb{P}^1,\mathcal{V}\big|_{\mathcal{U}_t})$ and a point $x\in \nu_t(\mathbb{P}^1)\cap B^\circ$, we can find a holomorphic family $\{\nu_s\}_{s\in\Delta}$ of rational curves in $\mathcal{K}$ that pass through $x$, with tangent vectors $\alpha_s$ at $x$ (with $\Delta\subset\mathcal{K}$ a small disc in some chart centered at our original point $t\in\mathcal{K}$), and such that $\frac{d}{ds}\big|_{s=t}\alpha_s=v(x).$ Let $w=\frac{d}{ds}\nu_s$ be the first-order deformation (holomorphic) vector field on this family. When restricted to each $\mathcal{U}_s$, $w$ is a section of
$$\nu_s^*TB^\circ=\mathcal{V}\big|_{\mathcal{U}_s}\oplus \mathcal{N}\big|_{\mathcal{U}_s}\cong\bigoplus_i \mathcal{O}(a_i)\oplus\mathcal{O}^{\oplus\ell},$$
and since $w(x)=0$, it must be a section of the $\bigoplus_i \mathcal{O}(a_i)$ factors, namely a section of $\mathcal{V}\big|_{\mathcal{U}_s}$. Pick a smooth family $U$ of $1$-forms on this family, i.e. a $C^\infty$ section of the relative cotangent bundle, with $u_s:=U|_{\mathcal{U}_s}\in H^0(\mathbb{P}^1,\mathcal{V}^\sharp\big|_{\mathcal{U}_s})$, and with $u_t=u$. Then by definition along $\nu_s$ we have
$$\iota_{w}U\big|_{\mathcal{U}_s}\equiv 0,$$
for all $s\in\Delta$,
and so along $\nu_t$ we have
$$L_wU\big|_{\mathcal{U}_t}=(d\iota_{w}U)\big|_{\mathcal{U}_t}+(\iota_w dU)\big|_{\mathcal{U}_t}=(\iota_w dU)\big|_{\mathcal{U}_t},$$
which vanishes at $x$ since $w(x)=0$.

We now use this to prove \eqref{cruxa2}, which by \eqref{curv} implies \eqref{cruxa}. For this, let $\zeta_s,s\in\Delta,$ be the smooth section of $\mathcal{N}\big|_{\mathcal{U}_s}$ over $\mathbb{P}^1\backslash\nu_s^{-1}(D)$ which maps to $u_s$ under \eqref{isow}, and recall that from \eqref{crux2} at $x$ we have
$$\Xi_x(\alpha_s,\zeta_s,\beta)=0,$$
for all $s\in\Delta$. Taking $\frac{d}{ds}\big|_{s=t}$ of this, we get
\begin{equation}\label{altern}
0=\Xi_x(v,\zeta,\beta)+\Xi_x(\alpha,L_w\zeta,\beta).
\end{equation}
Now at $x$ we have that $L_w\zeta$ is the vector that maps to $L_wU$ under \eqref{isow}, since at $x$ the metric $g_{\rm SK}$ does not get differentiated as it does not depend on $s$. Since we have shown that $(L_wU)(x)=0$, we deduce that $(L_w\zeta)(x)=0$,
and so \eqref{cruxa2} follows from \eqref{altern}.

(d) Given a section $u\in H^0(\mathbb{P}^1,\nu_t^*\mathcal{V}^\sharp)$, an analogous computation as in (a) gives
\begin{equation}\label{xjp2}\begin{split}
0=\ddbar|u|^2_h&=|\nabla u|^2_h-\langle R_{h}(u),u\rangle_h=|\nabla u|^2_h,
\end{split}\end{equation}
and so we conclude that $\nabla u=0$ on $\mathbb{P}^1\backslash\nu_t^{-1}(D)$.

(e) This is a direct consequence of part (d) and \cite[Prop. 1.4.18]{Koba}.
\end{proof}

Given $x\in\mathcal{U}^\circ$ and $v\in\mathcal{V}_x, \zeta\in\mathcal{N}_x$, recall from \eqref{ancorasplit} that
\begin{equation}\label{ankorasplit}
\mathcal{V}_x\oplus \mathcal{N}_x=T_{\mu(x)}B^\circ,
\end{equation}
so we can view $v$ and $\zeta$ also as tangent vectors in $B^\circ$. With this in mind, we have the following useful corollary:

\begin{corollary}\label{koro}
Let $x\in\mathcal{U}^\circ$, and let $v\in\mathcal{V}_x, \zeta\in\mathcal{N}_x$. Then at $\mu(x)\in B^\circ$ the curvature of the metric $g_{\rm SK}$ satisfies
\begin{equation}\label{cruxb}
R_{v\ov{v}\zeta\ov{\zeta}}=0,
\end{equation}
as well as
\begin{equation}\label{cruxb2}
\Xi(v,\zeta,\beta)=0,\quad\text{for all }\beta\in T_{\mu(x)}B^\circ.
\end{equation}
\end{corollary}
\begin{proof}
Let $t\in \mathcal{K}$ be such that the corresponding rational curve $\mathcal{U}_t$ contains $x$, and as usual denoted by $\nu_t:\mathbb{P}^1\to B$ the corresponding morphism. Since $\mathcal{V}\big|_{\mathcal{U}_t}\cong\oplus_i \mathcal{O}(a_i),$ $a_i>0,$ is a globally generated vector bundle, we can find a global section $V\in H^0(\mathbb{P}^1,\mathcal{V}\big|_{\mathcal{U}_t})$ such that $V(x)=v$. Let then $u\in\mathcal{V}^\sharp_x$ be the covector which is the image of $\zeta$ under \eqref{isow}. Since  $\mathcal{V}^\sharp\big|_{\mathcal{U}_t}\cong\mathcal{O}^{\oplus\ell}$ is a trivial vector bundle, we can find a global section $U\in H^0(\mathbb{P}^1,\mathcal{V}^\sharp\big|_{\mathcal{U}_t})$ such that $U(x)=u$. Then Theorem \ref{rigid} (c) applies to $U$ and $V$, and \eqref{cruxb}, \eqref{cruxb2} follow from \eqref{cruxa}, \eqref{cruxa2}.
\end{proof}

\section{The Ricci curvature in the direction of \texorpdfstring{$\mathcal{N}$}{N}}\label{se5}
Given $x\in \mathcal{U}^\circ$ and vectors $v\in\mathcal{V}_x, \zeta\in\mathcal{N}_x$ (which we can also view as tangent vectors in $T_{\mu(x)}B^\circ$ using \eqref{ankorasplit}), Corollary \ref{koro} shows that at $\mu(x)$ the Riemann curvature tensor of $g_{\rm SK}$ satisfies
$$R_{v\ov{v}\zeta\ov{\zeta}}=0.$$
As customary, we define the ``rough Laplacian'' of the Riemann curvature tensor of $g_{\rm SK}$, evaluated on $v,\zeta$ by
$$\Delta R_{v\ov{v}\zeta\ov{\zeta}}=\frac{1}{2}\left(\sum_i\nabla_i \nabla_{\ov{i}}R_{v\ov{v}\zeta\ov{\zeta}}+\sum_i \nabla_{\ov{i}}\nabla_iR_{v\ov{v}\zeta\ov{\zeta}}\right),$$
where $\{e_i\}$ is a local unitary frame.

The following is the main result of this section:

\begin{theorem}\label{sudore}
Given  $x\in \mathcal{U}^\circ$ and  $v\in\mathcal{V}_x, \zeta\in\mathcal{N}_x$, then at $\mu(x)$ we have
\begin{equation}\label{rattus}
\Delta R_{v\ov{v}\zeta\ov{\zeta}}=0,\quad R_{v\ov{\zeta}\beta\ov{\gamma}}=0,\quad \text{for all }\beta,\gamma\in T_{\mu(x)}B^\circ.
\end{equation}
\end{theorem}

Let $t\in \mathcal{K}$ be such that the corresponding rational curve $\mathcal{U}_t$ contains $x$, and as usual denoted by $\nu_t:\mathbb{P}^1\to B$ the corresponding morphism. As in the proof of Corollary \ref{koro}, we can extend $v$ to a section $v\in H^0(\mathbb{P}^1,\mathcal{V}\big|_{\mathcal{U}_t})$ and we can find a section $u\in H^0(\mathbb{P}^1,\mathcal{V}^\sharp\big|_{\mathcal{U}_t})$ such that the image of $u$ under \eqref{isow} is a smooth section $\zeta\in \mathcal{N}|_{\mathcal{U}_t}$ over $\mathbb{P}^1\backslash\nu_t^{-1}(D)$ which extends the given vector $\zeta$. The Ricci curvature $R_{\zeta\ov{\zeta}}$ along this curve and evaluated at $\zeta$ will also be denoted by $\Ric_{g_{\rm SK}}(u,\ov{u})$, which is a smooth function on $\mathbb{P}^1\backslash\nu_t^{-1}(D)$.

We wish to show that $\Ric_{g_{\rm SK}}(u,\ov{u})$ is a constant function on $\mathbb{P}^1\backslash\nu^{-1}(D)$. We will proceed in steps.

\subsection{Subharmonicity of \texorpdfstring{$\Ric_{g_{\rm SK}}(u,\ov{u})$}{Ric} }To start, we prove the following:
\begin{proposition}\label{subh}
The function $\Ric_{g_{\rm SK}}(u,\ov{u})$ on $\mathbb{P}^1\backslash\nu_t^{-1}(D)$ is subharmonic.
\end{proposition}
\begin{proof}
On $B^\circ$  define for $0\leq s\ll 1$
$$g_s = g_{\rm SK} - s \Ric_{g_{\rm SK}}.$$
It is clear that given any compact $K\Subset B^\circ$ there is some $0<s_K\ll 1$ such that $g_s$ is a K\"ahler metric on $K$ for $0\leq s\leq s_K$.

Standard direct computations (cf. \cite[P.185]{Mok}) show that given any $x\in B^\circ$ and two nonzero $(1,0)$ tangent vectors $v,\zeta$ at $x$,
we have the evolution equation at $x$ and $s=0$ for the bisectional curvature of $g_s$ evaluated along $v$ and $\zeta$
\begin{equation}\label{evolviti}
\frac{\de}{\de s}\bigg|_{s=0}R(g_s)_{v\ov{v}\zeta\ov{\zeta}}=\Delta R_{v\ov{v}\zeta\ov{\zeta}}+F(R)_{v\ov{v}\zeta\ov{\zeta}},
\end{equation}
where, as in Mok \cite{Mok}, we define
$$F(R)_{v\ov{v}\zeta\ov{\zeta}}=\sum_{\mu,\nu}R_{v\ov{v}\mu\ov{\nu}}R_{\zeta\ov{\zeta}\nu\ov{\mu}}-\sum_{\mu,\nu}|R_{v\ov{\mu}\zeta\ov{\nu}}|^2
+\sum_{\mu,\nu}|R_{v\ov{\zeta}\mu\ov{\nu}}|^2-\mathrm{Re}\left(R_{v\ov{\mu}}R_{\mu\ov{v}\zeta\ov{\zeta}}+R_{\zeta\ov{\mu}}R_{v\ov{v}\mu\ov{\zeta}}\right).$$
Equation \eqref{evolviti} is identical to the corresponding evolution of the bisectional curvature in the directions $v,\zeta$ along the K\"ahler-Ricci flow, see \cite{Mok}. Thanks to the crucial Lemma \ref{zat} below, we see that
\begin{equation}\label{final}
\frac{\de}{\de s}\bigg|_{s=0}R(g_s)_{v\ov{v}\zeta\ov{\zeta}}(x)\geq 0.
\end{equation}
Equip $\nu_t^*\Omega^1_B$ over the compact set $\nu_t^{-1}(K)$ with the Hermitian metric $h_s$ induced by $g_s$. At any point $y\in \nu_t^{-1}(K)$ for $0\leq s\leq s_K$, using the argument in \eqref{xjp}, we have
\begin{equation}\label{sat}
\ddbar\log|u|^2_{h_s}+\frac{\langle R_{h_s}(u),u\rangle_{h_s}}{|u|^2_{h_s}}\geq 0.
\end{equation}
We know from Theorem \ref{rigid} (d), that $u$ is parallel with respect to $h_0$ (the metric induced by $g_{\rm SK}$), hence (assuming without loss that $u$ is nontrivial) we can scale and assume without loss that $|u|^2_{h_0}\equiv 1$ on $\mathbb{P}^1\backslash\nu_t^{-1}(D)$. On the other hand, from Theorem \ref{rigid} (a), we know that  \eqref{new1} holds, and so $$\langle R_{h_0}(u),u\rangle_{h_0}=0.$$ Thus the LHS of \eqref{sat} vanishes at $y$ for $s=0$ and is nonnegative for $0\leq s\leq s_K$, hence at $y$ we have
\begin{equation}\label{zatt}\begin{split}
0&\leq \frac{\de}{\de s}\bigg|_{s=0}\left(\ddbar\log|u|^2_{h_s}+\frac{\langle R_{h_s}(u),u\rangle_{h_s}}{|u|^2_{h_s}}\right)\\
&=\ddbar\left(\frac{\de}{\de s}\bigg|_{s=0}|u|^2_{h_s}\right)+\frac{\de}{\de t}\bigg|_{t=0}\langle R_{h_s}(u),u\rangle_{h_s},
\end{split}\end{equation}
and writing $u=u_j dz^j$ and $|u|^2_{h_s}=u_i \ov{u_{j}}g_s^{i\ov{j}}$, observe that
$$\frac{\de}{\de s}\bigg|_{s=0}\left(u_i \ov{u_{j}}g_s^{i\ov{j}}\right)=-u_i \ov{u_{j}}g_{\rm SK}^{i\ov{s}}g_{\rm SK}^{r\ov{j}}\frac{\de}{\de s}\bigg|_{s=0}g_{s,r\ov{s}}
=u_i \ov{u_{j}}g_{\rm SK}^{i\ov{s}}g_{\rm SK}^{r\ov{j}}R_{r\ov{s}}=\Ric_{g_{\rm SK}}(u,\ov{u}).$$
Furthermore, we can write
$$\langle R_{h_s}(u),u\rangle_{h_s}=-R(g_s)_{v\ov{v}i\ov{j}}g_s^{i\ov{q}} g_s^{p\ov{j}}u_p\ov{u_q},$$
so
$$\frac{\de}{\de s}\bigg|_{s=0}\langle R_{h_s}(u),u\rangle_{h_s}=-\frac{\de}{\de s}\bigg|_{s=0}R(g_s)_{v\ov{v}\zeta\ov{\zeta}}-
R_{v\ov{v}i\ov{\zeta}}R_{\zeta\ov{i}}-R_{v\ov{v}\zeta\ov{i}}R_{\ov{\zeta}i},$$
but the last two terms vanish since using \eqref{curv} and \eqref{cruxa2}, we can write
$$R_{v\ov{v}i\ov{\zeta}}=\Xi_{v i q}\ov{\Xi_{v\zeta q}}=0, \quad R_{v\ov{v}\zeta\ov{i}}=\Xi_{v\zeta q}\ov{\Xi_{v i q}}=0,$$
and putting these all together gives
\begin{equation}\label{zatt2}\begin{split}
0&\leq\ddbar\left(\Ric_{g_{\rm SK}}(u,\ov{u})\right)-\frac{\de}{\de s}\bigg|_{s=0}R(g_s)_{v\ov{v}\zeta\ov{\zeta}}\\
&\leq \ddbar\left(\Ric_{g_{\rm SK}}(u,\ov{u})\right),
\end{split}\end{equation}
using \eqref{final}. Since  $K\Subset B^\circ$ is arbitrary, this shows that the function $\Ric_{g_{\rm SK}}(u,\ov{u})$ is subharmonic on $\mathbb{P}^1\backslash \nu_t^{-1}(D)$.
\end{proof}

We used the following lemma, which is the analog of ``condition $(\sharp)$'' in Mok, but the proof here is substantially easier:

\begin{lemma}\label{zat} In the setting of Theorem \ref{sudore}, at $\mu(x)$ we have
$$\Delta R_{v\ov{v}\zeta\ov{\zeta}}\geq 0,\quad F(R)_{v\ov{v}\zeta\ov{\zeta}}\geq 0.$$
\end{lemma}
\begin{proof}
Recall from \eqref{curv} that
$$R_{i\ov{j}k\ov{\ell}}=g_{\rm SK}^{p\ov{q}}\Xi_{ikp}\ov{\Xi_{j\ell q}}.$$
From \eqref{cruxb} we then see that at $\mu(x)$ we have
$$0=R_{v\ov{v}\zeta\ov{\zeta}}=\sum_\beta|\Xi_{v\zeta\beta}|^2,$$
and so $\Xi_{v\zeta\beta}(x)=0$ for all $\beta\in T_{\mu(x)}B^\circ$, and furthermore for all $\mu,\nu\in T_{\mu(x)}B^\circ$
$$0=\sum_\beta \Xi_{v\zeta\beta}\ov{\Xi_{\mu\nu\beta}}=R_{v\ov{\mu}\zeta\ov{\nu}}.$$

Now take the definition of $\Delta R$ and use \eqref{curv} and the fact that $\Xi$ is holomorphic to get
$$\Delta R_{v\ov{v}\zeta\ov{\zeta}}=\sum_{i,p}|\nabla_i \Xi_{v\zeta p}|^2+\sum_{i,p}\mathrm{Re}(\ov{\Xi_{v\zeta p}}\nabla_{\ov{i}}\nabla_i\Xi_{v\zeta p})=\sum_{i,p}|\nabla_i \Xi_{v\zeta p}|^2\geq 0,$$
since $\Xi_{v\zeta p}(x)=0$.
For the $F(R)$ term, from its definition we see that at $\mu(x)$ we have
$$F(R)_{v\ov{v}\zeta\ov{\zeta}}=\sum_{\mu,\nu}R_{v\ov{v}\mu\ov{\nu}}R_{\zeta\ov{\zeta}\nu\ov{\mu}}
+\sum_{\mu,\nu}|R_{v\ov{\zeta}\mu\ov{\nu}}|^2.$$
As in Mok \cite[(7)]{Mok}, if we pick $\{e_\mu\}$ a unitary basis of eigenvectors of the Hermitian form $H_v(\mu,\nu)=R_{v\ov{v}\mu\ov{\nu}}$, then in this basis we see that
\begin{equation}\label{musmusculus}
F(R)_{v\ov{v}\zeta\ov{\zeta}}=\sum_{\mu}R_{v\ov{v}\mu\ov{\mu}}R_{\zeta\ov{\zeta}\mu\ov{\mu}}
+\sum_{\mu,\nu}|R_{v\ov{\zeta}\mu\ov{\nu}}|^2\geq 0.
\end{equation}
\end{proof}

\subsection{Constancy of \texorpdfstring{$\Ric_{g_{\rm SK}}(u,\ov{u})$}{Ric} }The next step is the following:

\begin{proposition}\label{bdd}
The function $\Ric_{g_{\rm SK}}(u,\ov{u})$ on $\mathbb{P}^1\backslash\nu_t^{-1}(D)$ is constant.
\end{proposition}
\begin{proof}
Since the function $\Ric_{g_{\rm SK}}(u,\ov{u})$ on $\mathbb{P}^1\backslash\nu_t^{-1}(D)$ is subharmonic by Proposition \ref{subh}, it suffices to show that it is bounded.

Recall that, using \eqref{ankorasplit}, our sections $v,\zeta$ can be viewed as vector fields along $\nu_t(\mathbb{P}^1)\cap B^\circ$.
Our first claim is that for every $y\in\nu_t(\mathbb{P}^1)\cap B^\circ$ and local sections $v$ of $\mathcal{V}$ and $\zeta$ of $\mathcal{N}$ near $y$, we have
\begin{equation}\label{perp}
R_{v\ov{\zeta}}=0.
\end{equation}
Indeed, recall from \eqref{curv} that
$$R_{v\ov{\zeta}}=\sum_{p,q} \Xi_{v pq}\ov{\Xi_{\zeta pq}},$$
where $\{e_p\}$ is a $g_{\rm SK}$-unitary frame at our point $y$. Since $\mu^*TB^\circ=\mathcal{V}\oplus\mathcal{N}$, we may choose the frame so that
$e_j\in\mathcal{V}$ for $1\leq j\leq n-\ell,$ and $e_j\in\mathcal{N}$ for $n-\ell+1\leq j\leq n$. Recalling from \eqref{cruxb2} that $\Xi_{uvw}=0$ whenever $u\in\mathcal{V}$ and $v\in\mathcal{N}$, we see that $\Xi_{v pq}=0$ except possibly when $1\leq p,q\leq n-\ell$, so that
$$R_{v\ov{\zeta}}=\sum_{p,q=1}^{n-\ell} \Xi_{v pq}\ov{\Xi_{\zeta pq}}=0,$$
since $\Xi_{\zeta pq}=0$ when  $1\leq p,q\leq n-\ell$, proving our claim.

Recall that, as explained earlier, we may assume that $\nu_t(\mathbb{P}^1)$ intersects $D$ only at regular points of $D$ and that these intersections are transverse.
To prove the boundedness of $\Ric_{g_{\rm SK}}(u,\ov{u})$ it suffices to prove near any of the finitely many points in $\nu_t^{-1}(D)$. Let $y$ be such a point, and choose an open neighborhood $U$ of $z=\nu_t(y)$ in $B$ with local holomorphic coordinates centered at $z$ such that $D\cap U=\{z_1=0\}$ and $\nu_t(\mathbb{P}^1)\cap U=\{z_2=\cdots=z_n=0\}$, so that $\de_1$ is tangent to the rational curve while $\de_2,\dots,\de_n$ are tangent to $D$. We will work on $\nu_t^{-1}(U\cap\{z_1\neq 0\})$ which in our chart is identified with $\{z_1\neq 0, z_2=\cdots=z_n=0\}=:V$.

Thanks to Proposition \ref{riccicurb} we know that on $V$ we have
\begin{equation}\label{ricci3}
0\leq R_{i\ov{i}}\leq C,\quad 2\leq i\leq n,
\end{equation}
\begin{equation}\label{ricci4}
0\leq R_{1\ov{1}}\leq \frac{C}{|z_1|^2}.
\end{equation}

Using \eqref{sl}, together with the fact that $u$ is a holomorphic section on all of $\mathbb{P}^1$, we see that
\begin{equation}\label{l1}
\sup_{V}|\zeta|^2_{\nu_t^*g_B}\leq C\sup_{V}|u|^2_{\nu_t^*g_B}<\infty.
\end{equation}
In our coordinates we can write
$$\zeta=\zeta^1\de_1+\sum_{j\geq 2}\zeta^j\de_j=:\zeta^1\de_1+\zeta_D,$$
and the function $\zeta^1$ is equal to $\langle dz_1,u\rangle_{g_{\rm SK}}$. From \eqref{l1} we see that
\begin{equation}\label{l2}
\sup_{V}|\zeta_D|^2_{\nu_t^*g_{B}}<\infty,
\end{equation}
and from this and \eqref{ricci3} we see that on $V$ we have
\begin{equation}\label{l3}
0\leq  R_{\zeta_D\ov{\zeta_D}}\leq C.
\end{equation}

Since $\de_1$ is the tangent vector to $\nu_t(\mathbb{P}^1)$, it belongs to $\mathcal{V}\big|_{\mathcal{U}_t}$. On the other hand $\zeta$ belongs to $\mathcal{N}\big|_{\mathcal{U}_t}$, hence \eqref{perp} (restricted to $V$) gives
$$0=R_{1\ov{\zeta}}(z_1)=\ov{\zeta^1(z_1)}R_{1\ov{1}}(z_1)+R_{1\ov{\zeta_D}}(z_1),$$
and since $\Ric_{g_{\rm SK}}\geq 0$ on $\{z_1\neq 0\}$, Cauchy-Schwarz together with \eqref{ricci3} and \eqref{l3} give
$$|\zeta^1(z_1)| R_{1\ov{1}}(z_1)=|R_{1\ov{\zeta_D}}(z_1)|\leq R_{1\ov{1}}(z_1)^{\frac{1}{2}} R_{\zeta_D\ov{\zeta_D}}(z_1)^{\frac{1}{2}}\leq CR_{1\ov{1}}(z_1)^{\frac{1}{2}},$$
i.e.
$$|\zeta^1(z_1)|  R_{1\ov{1}}(z_1)^{\frac{1}{2}}\leq C,$$
and using again that  $\Ric_{g_{\rm SK}}\geq 0$, together with \eqref{l3} we can estimate
\[\begin{split}
0\leq R_{\zeta\ov{\zeta}}(z_1)&\leq C|\zeta^1(z_1)|^2 R_{1\ov{1}}(z_1)+C R_{\zeta_D\ov{\zeta_D}}(z_1)\leq C,
\end{split}\]
as desired.
\end{proof}

We can now conclude the proof of Theorem \ref{sudore}, by showing that at $\mu(x)$ we have
\begin{equation}\label{zazz}\Delta R_{v\ov{v}\zeta\ov{\zeta}}=0,\quad F(R)_{v\ov{v}\zeta\ov{\zeta}}= 0.
\end{equation}
Indeed, Proposition \ref{bdd} shows that the function $\Ric_{g_{\rm SK}}(u,\ov{u})$ on $\mathbb{P}^1\backslash\nu_t^{-1}(D)$ is constant, hence going back to \eqref{zatt2} and recalling \eqref{final} and \eqref{evolviti} shows that
$$0=\frac{\de}{\de s}\bigg|_{s=0}R(g_s)_{v\ov{v}\zeta\ov{\zeta}}=\Delta R_{v\ov{v}\zeta\ov{\zeta}}+F(R)_{v\ov{v}\zeta\ov{\zeta}}.$$
Recalling Lemma \ref{zat}, we see that \eqref{zazz} holds. To finally deduce from \eqref{zazz} that the last equality in \eqref{rattus} holds, it suffices to plug in the fact that $F(R)_{v\ov{v}\zeta\ov{\zeta}}= 0$ into \eqref{musmusculus}, and see that
$$\sum_{\mu,\nu}|R_{v\ov{\zeta}\mu\ov{\nu}}|^2=0,$$
as desired.

\section{Constructing a parallel subbundle of \texorpdfstring{$\mu^*TB^\circ$}{TB}}\label{se6}
Recall that above we have constructed a decomposition $\mu^*TB^\circ=\mathcal{V}\oplus \mathcal{N}$ over $\mathcal{U}^\circ$, where $\mathcal{V}$ is a nontrivial holomorphic subbundle, which is not equal to $\mu^*TB^\circ$ whenever $B\not\cong\mathbb{P}^n$.

The following is then our main theorem (Theorem \ref{mainn}):
\begin{theorem}\label{parall}
The holomorphic subbundle $\mathcal{V}\subset \mu^*TB^\circ$ over $\mathcal{U}^\circ$ is preserved by $\nabla$, the pullback of the Levi-Civita connection of $\omega_{\rm SK}$.
\end{theorem}

Recall that by Theorem \ref{prolisso2} $f$ is either of maximal variation or isotrivial. The proof of Theorem \ref{parall} will be quite different in these two cases.

Observe that after Theorem \ref{parall} is proved, it follows that the orthogonal complement $\mathcal{N}\subset \mu^*TB^\circ$ is also a holomorphic subbundle, preserved by $\nabla$, see e.g. \cite[Prop. 1.4.18]{Koba}, and the same holds for their duals $\mathcal{V}^\sharp,\mathcal{N}^\sharp\subset\mu^*\Omega^1_{B^\circ}$.

\subsection{Maximal Variation Case}
In this section we give the proof of Theorem \ref{parall} in the case when $f$ has maximal variation. Recall from Corollary \ref{prolisso} that in this case $g_{\rm SK}$ has positive Ricci curvature on $B^\circ$.

We work at a point $x\in \mathcal{U}^\circ$. Let $v$ be a local holomorphic section of $\mathcal{V}$ near $x$,
and let $\gamma:(-\ve,\ve)\to \mathcal{U}^\circ$ be a smooth curve with $\gamma(0)=x,\dot{\gamma}(0)=\eta\neq 0$. The goal of Theorem \ref{parall} is then to show that $\nabla_\eta v\in\mathcal{V}$. Using the decomposition $\mu^*TB^\circ=\mathcal{V}\oplus\mathcal{N}$, we can write
$$\nabla_\eta v=-\xi-\zeta,\quad \xi\in \mathcal{V}_x, \zeta\in\mathcal{N}_x,$$
(the minus sign is only to match the notation in Mok \cite{Mok}), so we wish to show that $\zeta=0$.
The following argument is a modification of a result of Mok \cite[Proposition 3.1']{Mok}, specifically of equation (21) on p.211:
\begin{proposition}
At $\mu(x)$ we have
\begin{equation}\label{vanish}
R_{\zeta\ov{\zeta}\zeta'\ov{\zeta'}}=0,
\end{equation}
for all $\zeta'\in\mathcal{N}_x$.
\end{proposition}
Here and in the following we are again using \eqref{ankorasplit} to view $\zeta,\zeta'$ also as tangent vectors in $T_{\mu(x)}B^\circ$.
Also, since $\nabla$ is the pullback connection, when taking $\nabla_v$ for some $v\in T\mathcal{U}^\circ$ it is really only $\mu_*(v)\in TB^\circ$ that enters.
\begin{proof}
For $t\in (-\ve,\ve)$, let $\beta(t)$ be the parallel transport of $v(x)$ along $\gamma$, let $v(t)=v|_{\gamma(t)}$, and define $\xi(t),\zeta(t)$ by
$$\beta(t)=v(t)+t\xi(t)+t\zeta(t),\quad \xi(t)\in \mathcal{V}_{\gamma(t)}, \zeta(t)\in\mathcal{N}_{\gamma(t)},$$
so that
$$0=\nabla_\eta\beta(0)=\nabla_\eta v+\xi(0)+\zeta(0),$$
and so we see that $\xi(0)=\xi,\zeta(0)=\zeta$. Given an arbitrary $\zeta'\in\mathcal{N}_x$, let $\chi(t)$ be the parallel transport of $\zeta'$ along $\gamma$, so that $\chi(0)=\zeta'$ and $\nabla_{\dot{\gamma}(t)}\chi(t)=0$. We can also write
$$\chi(t)=\zeta'(t)+t\theta(t),\quad \zeta'(t)\in \mathcal{N}_{\gamma(t)}, \theta(t)\in\mathcal{V}_{\gamma(t)},$$
and $\zeta'(0)=\zeta'$. We can expand at the point $\gamma(t)$
\[\begin{split}
R_{\beta(t)\ov{\beta(t)}\chi(t)\ov{\chi(t)}}&=R_{v\ov{v}\zeta'\ov{\zeta'}}+t\left(2\mathrm{Re}R_{v\ov{v}\zeta'\ov{\theta}}+2\mathrm{Re}R_{v\ov{\xi}\zeta'\ov{\zeta'}}+
2\mathrm{Re}R_{v\ov{\zeta}\zeta'\ov{\zeta'}}\right)\\
&+t^2\Big(R_{v\ov{v}\theta\ov{\theta}}+2\mathrm{Re}R_{v\ov{\xi}\zeta'\ov{\theta}}+2\mathrm{Re}R_{v\ov{\zeta}\zeta'\ov{\theta}}
+2\mathrm{Re}R_{v\ov{\xi}\theta\ov{\zeta'}}+2\mathrm{Re}R_{v\ov{\zeta}\theta\ov{\zeta'}}\\
&\quad\quad+R_{\xi\ov{\xi}\zeta'\ov{\zeta'}}+R_{\zeta\ov{\zeta}\zeta'\ov{\zeta'}}+2\mathrm{Re}R_{\xi\ov{\zeta}\zeta'\ov{\zeta'}}\Big)+O(t^3),
\end{split}\]
where $O(t^3)$ denotes a vector-valued function of length bounded above by $Ct^3$. Recalling \eqref{curv}, we can express the curvature tensor in terms of $\Xi$, and since $\Xi(v,\zeta',\beta)=\Xi(\xi,\zeta',\beta)=0$ for all $\zeta'\in\mathcal{N}_x$ and all $\beta\in T_{\mu(x)}B^\circ$ (by Corollary \ref{koro}), many terms in this expansion vanish. Using furthermore that $R_{v\ov{\zeta}\beta\ov{\delta}}=0$ for all $\beta,\delta\in T_{\mu(x)}B^\circ$ (by Theorem \ref{sudore}), the expression finally reduces to
\[\begin{split}
R_{\beta(t)\ov{\beta(t)}\chi(t)\ov{\chi(t)}}&=t^2\Big(R_{v\ov{v}\theta\ov{\theta}}+R_{\zeta\ov{\zeta}\zeta'\ov{\zeta'}}\Big)+O(t^3).
\end{split}\]
Defining (similarly to Mok)
$$A=R_{v\ov{v}\theta\ov{\theta}}+R_{\zeta\ov{\zeta}\zeta'\ov{\zeta'}},$$
and since the bisectional curvature is nonnegative, we have $R_{v\ov{v}\theta\ov{\theta}}\geq 0$, and so
\begin{equation}\label{ena}
A\geq R_{\zeta\ov{\zeta}\zeta'\ov{\zeta'}}.
\end{equation}
At this point notice that
\begin{equation}\label{dva}
A=\frac{1}{2}\frac{d^2}{dt^2}\bigg|_{t=0}R_{\beta(t)\ov{\beta(t)}\chi(t)\ov{\chi(t)}}=\nabla^2_{\eta\eta}R_{v\ov{v}\zeta'\ov{\zeta'}},
\end{equation}
using that $\beta(t),\chi(t)$ are parallel along $\gamma$ and that $\nabla$ is a pullback connection. On the other hand we claim that at $x$ we have
$$\nabla^2_{ww}R_{v\ov{v}\zeta'\ov{\zeta'}}\geq 0,$$ for all real tangent vectors $w$ at $x$. Indeed, pick a curve in $\mathcal{U}^\circ$ passing through $x$ and tangent to $w$, and let $\ti{v}(t),\ti{\zeta}'(t)$ be the parallel transport of $v,\zeta'$ along this curve, then
$R_{\ti{v}(t)\ov{\ti{v}(t)}\ti{\zeta}'(t)\ov{\ti{\zeta}'(t)}}\geq 0,$ and $R_{v\ov{v}\zeta'\ov{\zeta'}}=0$ by Corollary \ref{koro}, and so
$$0\leq \frac{d^2}{dt^2}\bigg|_{t=0}R_{\ti{v}(t)\ov{\ti{v}(t)}\ti{\zeta}'(t)\ov{\ti{\zeta}'(t)}}=\nabla^2_{ww}R_{v\ov{v}\zeta'\ov{\zeta'}},$$
as claimed. But recall that Theorem \ref{sudore} showed that $\Delta R_{v\ov{v}\zeta'\zeta'}=0$, and since this is an average of terms of the form  $\nabla^2_{ww}R_{v\ov{v}\zeta'\ov{\zeta'}}$ as $\mu_*(w)$ varies among all $g_{\rm SK}$-unit tangent vectors at $\mu(x)$, we see that necessarily $\nabla^2_{ww}R_{v\ov{v}\zeta'\ov{\zeta'}}=0$ for all $w$. Using \eqref{ena} and \eqref{dva} we get
$$0=\nabla^2_{\eta\eta}R_{v\ov{v}\zeta'\ov{\zeta'}}=2A\geq R_{\zeta\ov{\zeta}\zeta'\ov{\zeta'}}\geq 0,$$
which proves \eqref{vanish}.
\end{proof}

Now that \eqref{vanish} is established, we can show that $\zeta=0$ as follows: combining \eqref{vanish} with \eqref{curv} gives
$$\Xi(\zeta,\zeta',\beta)=0,$$
for all $\beta\in T_{\mu(x)}B^\circ$ and all $\zeta'\in\mathcal{N}_x$. But thanks to Corollary \ref{koro} we also have
$$\Xi(\zeta,\mu,\beta)=0,$$
for all $\beta\in T_{\mu(x)}B^\circ$ and all $\mu\in\mathcal{V}_x,$ and since $T_{\mu(x)}B^\circ\cong \mathcal{V}_x\oplus\mathcal{N}_x$, it follows that
$$\Xi(\zeta,\mu,\beta)=0,$$
for all $\mu, \beta\in T_{\mu(x)}B^\circ$. From the formula for the curvature tensor,
$$\Ric_{g_{\rm SK}}(\zeta,\ov{\zeta})=\sum_{p,q}|\Xi(\zeta,e_p,e_q)|^2=0.$$
Since we assume $f$ of maximal variation, $\Ric_{g_{\rm SK}}>0$ on $B^\circ$, and so $\zeta=0$. This concludes the proof of Theorem \ref{parall} when $f$ has maximal variation.

\subsection{Isotrivial Case}
In this section we give the proof of Theorem \ref{parall} in the case when $f$ is isotrivial, and so $\omega_{\rm SK}$ is flat by Corollary \ref{prolisso}. We wish to show that the subbundle $\mathcal{V}\subset \mu^*TB^\circ$ is parallel under $\nabla$, and by duality this is equivalent to showing that $\mathcal{V}^\sharp\subset \mu^*\Omega^1_{B^\circ}$ is parallel under $\nabla$. Recall that $\rho:\mathcal{U}\to\mathcal{K}$ is a $\mathbb{P}^1$-bundle. Thus, given $x\in \mathcal{U}^\circ$ and $v\in T_x\mathcal{U}^\circ$, we can decompose $T_x\mathcal{U}^\circ$ as the direct sum of the tangent line to the vertical $\mathbb{P}^1$ direction and a complementary subspace, and thus write
$v=v_1+v_2,$ where $v_1$ is tangent to a rational curve $\mathcal{U}_t$ (for some $t\in\mathcal{K}$) that contains $x$ (which on $\mathcal{U}_t$ corresponds to a point $y\in\mathbb{P}^1$) and $v_2$ is transverse to $\mathcal{U}_t$. The rational curve morphism will be as usual denoted by $\nu_t:\mathbb{P}^1\to B$. We may also assume that $\nu_t(\mathbb{P}^1)$ intersects $D$ only at regular points of $D$.
Since $\nu_t$ is free, we can deform it in a $1$-parameter family $\pi:\mathbb{P}^1\times\Delta\to B$,  with $s\in\Delta$ (where $\Delta\subset\mathcal{K}$ is a small disc in some chart centered at $t\in\mathcal{K}$), such that $\nu_s:=\pi(\cdot, s):\mathbb{P}^1\to B$ are rational curves in $\mathcal{K}$ which are also not contained in $D$ and such that the first order deformation vector $\frac{\de}{\de s}\big|_{s=t}\nu_s\in H^0(\mathbb{P}^1,\nu_t^*TB)$ agrees with $v_2$ at $x$. Up to shrinking $\Delta$, we have a natural inclusion $\sigma:\mathbb{P}^1\times \Delta\hookrightarrow \mathcal{U}$ such that $\mu\circ\sigma=\pi$. The intersection $\sigma(\mathbb{P}^1\times \Delta)\cap \mathcal{U}^\circ$ is Zariski open in $\sigma(\mathbb{P}^1\times \Delta)$ and contains the point $(y,0)$.

We then choose a smooth $(1,0)$ vector field $V$ on $\mathbb{P}^1\times\Delta$ which restricted to $\mathbb{P}^1\times\{0\}$ is the first order deformation vector, and so it satisfies
$d\sigma_{(y,0)}(V)=v_2$. To prove that $\mathcal{V}^\sharp$ is preserved by $\nabla_v$ at $x$, it will suffice to construct a smooth frame $u_1,\dots, u_\ell$ for $\sigma^*\mathcal{V}^\sharp$ over $\mathbb{P}^1\times\Delta$ such that
\begin{equation}\label{desire}
(\nabla_V u_i)(y,0)\in\sigma^*\mathcal{V}^\sharp_x,\quad 1\leq i\leq \ell,
\end{equation}
where $\nabla$ also denotes the pullback connection, since by Theorem \ref{rigid} (d) we have that along $\nu_t$
$$(\nabla_{v_1} u_i)(x)=0.$$

For every $s\in\Delta$, $\sigma^*\mathcal{V}^\sharp\big|_{\mathbb{P}^1\times \{s\}}$ is a trivial vector bundle of rank $\ell$ over $\nu_s$, which over $\mathbb{P}^1\backslash\nu_s^{-1}(D)$ is equipped with the metric induced by $\omega_{\rm SK}$. For each $s\in\Delta$ we can then choose a global holomorphic frame $u_1(s),\dots, u_\ell(s)\in H^0(\mathbb{P}^1, \sigma^*\mathcal{V}^\sharp\big|_{\mathbb{P}^1\times \{s\}})$, smoothly dependent on $s\in\Delta$. Thanks to Theorem \ref{rigid} (d), each $u_i(s)$ is parallel (with respect to the connection induced by $\omega_{\rm SK}$) over $\mathbb{P}^1\backslash\nu_s^{-1}(D)$. Varying $s$, these sections define a smooth frame $u_1,\dots,u_\ell$ of $\sigma^*\mathcal{V}^\sharp$ over $\mathbb{P}^1\times\Delta$, which is parallel when restricted to each $(\mathbb{P}^1\times\{s\})\cap\pi^{-1}(B^\circ)$.
Fix now any $1\leq i\leq \ell$,
and recall that
\begin{equation}\label{splitting}
\pi^*\Omega^1_{B^\circ}=\sigma^*\mathcal{V}^\sharp\oplus\sigma^*\mathcal{N}^\sharp,
\end{equation}
where $\mathcal{N}^\sharp$ is the annihilator of $\mathcal{N}\subset \mu^*TB^\circ$. Let $\mathcal{P}$ be the $g_{\rm SK}$-orthogonal projection onto the $\sigma^*\mathcal{N}^\sharp$ factor, which is defined on $\pi^{-1}(B^\circ)$ and consider
$$\mathcal{P}(\nabla_V u_i),$$
a smooth section of $\sigma^*\mathcal{N}^\sharp\subset\pi^*\Omega^1_{B^\circ}$ over $\pi^{-1}(B^\circ)$. Let also $\iota:\mathbb{P}^1\hookrightarrow \mathbb{P}^1\times \Delta$ be the embedding $z\mapsto (z,0)$, so $\pi\circ\iota=\nu_t$.

\begin{lemma}The pullback $\iota^*(\mathcal{P}(\nabla_V u_i))$ to $\mathbb{P}^1\backslash\nu_t^{-1}(D)$ is a parallel section of $\mathcal{N}^\sharp\big|_{\mathcal{U}_t}$ over $\mathbb{P}^1\backslash\nu_t^{-1}(D)$.
\end{lemma}
\begin{proof}We work at an arbitrary point in $\mathbb{P}^1\backslash\nu_t^{-1}(D)$, let $W$ be any local holomorphic vector field near our point which is tangent to the $\mathbb{P}^1$ factor.
Since the splitting $$\nu_s^*\Omega^1_{B^\circ}=\mathcal{V}^\sharp\big|_{\nu_s}\oplus\mathcal{N}^\sharp\big|_{\nu_s}$$ is preserved by $\nabla$ (by Theorem \ref{rigid} (e)), and since $g_{\rm SK}$ is flat, we have
\[\begin{split}
\nabla_{W}(\iota^*(\mathcal{P}(\nabla_V u_i)))&=\iota^*(\nabla_{W}(\mathcal{P}(\nabla_V u_i)))=
\iota^*(\mathcal{P}(\nabla_W\nabla_V u_i))\\
&=\iota^*(\mathcal{P}(\nabla_V \nabla_W u_i+\nabla_{[W,V]} u_i)).
\end{split}\]
Now, since $u_i$ is parallel along the rational curve $\nu_s(\mathbb{P}^1)\backslash D$ for all $s\in\mathbb{C}$, we have that $\nabla_W u_i$ vanishes identically on $U\times \Delta$ and so
$$\nabla_V \nabla_W u_i=0.$$
Furthermore,
$[W,V]=-L_VW$ is also tangent to $\nu_t(\mathbb{P}^1)$, so $\nabla_{[W,V]}u_i=0$ too.
\end{proof}

Since  $\iota^*(\mathcal{P}(\nabla_V u_i))$ is parallel, it is in particular holomorphic  over $\mathbb{P}^1\backslash\nu_t^{-1}(D)$. The following Lemma then implies that  $\iota^*(\mathcal{P}(\nabla_V u_i))$ extends to a holomorphic section of $\nu_t^*\Omega^1_B$ over $\mathbb{P}^1$:

\begin{proposition}
Let $w\in H^0(\mathbb{P}^1\backslash \nu_t^{-1}(D), \nu_t^*\Omega^1_B)$ be a holomorphic section which is parallel with respect to $\nabla$ (the Chern connection induced by $\omega_{\rm SK}$). Then $w$ extends to a holomorphic section of $\nu_t^*\Omega^1_B$ over all of $\mathbb{P}^1$.
\end{proposition}
\begin{proof}
Since $w$ is parallel, its pointwise length $|w|^2_{\nu_t^*g_{\rm SK}}$ is constant on $\mathbb{P}^1\backslash \nu_t^{-1}(D)$.
Recall that from \eqref{sl} we have that
\begin{equation}\label{f1}
\omega_{\rm SK}\geq C^{-1}\omega_B,
\end{equation}
on $B^\circ$. Since $\nu_t^{-1}(D)$ is a finite subset of $\mathbb{P}^1$, we consider the extension problem of $w$ across each of these points, so let $y\in\nu_t^{-1}(D)$ be one of them.
Recall that $D$ is regular at the point $x=\nu_t(y)$, and we can choose local holomorphic coordinates $z_1,\dots z_n$ on a chart $U$ centered at $x$ such that $D\cap U=\{z_1=0\}$.
The volume form $\omega_{\rm SK}^n$ is given by a fiber integration as in \eqref{fibber}, and its asymptotic behavior near $D$ is studied in \cite[Theorem 2.1]{GTZ3} (see also \cite{BJ, GTZ2} for the case when $\dim B=1$  and \cite{Ki}, \cite{Taka2} for $\dim B$ arbitrary) where it is shown that
\begin{equation}\label{f2}
\omega_{\rm SK}^n\leq \frac{C}{|z_1|^{2(1-\gamma)}}(-\log|z_1|)^C \omega_B^n,
\end{equation}
on $U\cap \{z_1\neq 0\}$, for some $C>0$ and $\gamma\in (0,1]$. Combining \eqref{f1} and \eqref{f2} gives the crude bound
\begin{equation}\label{f3}
\omega_{\rm SK}\leq \frac{C}{|z_1|^{2(1-\gamma)}}(-\log|z_1|)^C \omega_B,
\end{equation}
see also \cite[(2.1) and Theorem 3.4]{TZ} and \cite[Theorem 1.1]{GTZ3} for sharper and more general such bounds.
Passing to the dual metric on $\Omega^1_B$ and pulling back via $\nu_t$, \eqref{f3} implies that on the punctured neighborhood $\nu_t^{-1}(U\cap\{z_1\neq 0\})$ of $y$ in $\mathbb{P}^1$ we have
$$|w|^2_{\nu_t^*g_{B}}\leq \frac{C}{|z_1|^{2(1-\gamma)}}(-\log|z_1|)^C|w|^2_{\nu_t^*g_{\rm SK}}=\frac{C'}{|z_1|^{2(1-\gamma)}}(-\log|z_1|)^C,$$
and from this we see that $|w|^2_{\nu_t^*g_{B}}$ is $L^1$ in $\nu_t^{-1}(U\cap\{z_1\neq 0\})$. Since $\nu_t^*\Omega^1_B$ is a trivial bundle over $\nu_t^{-1}(U)$, we can represent $w$ locally as an $n$-tuple of holomorphic functions on $\nu_t^{-1}(U\cap\{z_1\neq 0\})$, and since these functions are in $L^2$, they extend holomorphically across the point $y$ (see e.g. \cite[Proposition 1.14]{Oh}), which gives us the desired extension of $w$.
\end{proof}

At this point we have shown that  $\iota^*(\mathcal{P}(\nabla_V u_i))$ gives a holomorphic section $w\in H^0(\mathbb{P}^1,\nu_t^*\Omega^1_B)$. Recalling the splitting
\eqref{h6}, we see that $w$ must be a section of the factor $\mathcal{O}^{\oplus \ell}$, i.e. a section of $\mathcal{V}^\sharp\big|_{\mathcal{U}_t}$. Since it is also a section of $\mathcal{N}^\sharp\big|_{\mathcal{U}_t}$, it must be identically zero. This shows that $\iota^*(\mathcal{P}(\nabla_V u_i))=0$, and so $\iota^*(\nabla_V u_i)\in \mathcal{V}^\sharp\big|_{\mathcal{U}_t}$, and so \eqref{desire} is established.
This concludes the proof of Theorem \ref{parall} when $f$ is isotrivial.

\section{Obtaining a parallel \texorpdfstring{$(1,1)$}{(1,1)}-form and Hwang's Theorem}\label{se7}
In this section we show how to combine our main theorem \ref{mainn} with results of Voisin \cite{Vo}, Hwang \cite{Hw2,Hw} and Bakker-Schnell \cite{BS} to deduce Theorem \ref{main}. The key step is the following:
\begin{theorem}\label{form}
Suppose that $B\not\cong\mathbb{P}^n$. Then there is a nontrivial real $(1,1)$-form $\psi$ on $B^\circ$ with $\nabla^{\rm SK}\psi=0$ and $\psi$ not proportional to $\omega_{\rm SK}$.
\end{theorem}

First, we show that Theorem \ref{main} follows from this (we do not need to assume that $X$ is projective):
\begin{proof}[Proof of Theorem \ref{main}]
Suppose for a contradiction that $B\not\cong\mathbb{P}^n$. Then by Theorem \ref{form} the $2$-forms $\omega_{\rm SK}$ and $\psi$ on $B^\circ$ are both $\nabla^{\rm SK}$-parallel and not proportional, and thus they give us a $2$-dimensional space of global sections of the local system $R^2f_*\mathbb{R}_{X^\circ}$ over $B^\circ$.
However, as observed by Voisin \cite[Lemma 5.5]{Vo}, a result of Matsushita \cite{Ma5} together with Deligne's invariant cycles theorem show that this space of sections is always $1$-dimensional, a contradiction.
\end{proof}

Since $B\not\cong\mathbb{P}^n$, we know that $\mathcal{V}\subset \mu^*TB^\circ$ is a nontrivial proper holomorphic subbundle over $\mathcal{U}^\circ$, which by Theorem \ref{parall} is preserved by $\nabla$. As mentioned after Theorem \ref{parall}, the $g_{\rm SK}$-orthogonal complement $\mathcal{N}\subset \mu^*TB^\circ$ of $\mathcal{V}$ is also a nontrivial proper holomorphic subbundle over $\mathcal{U}^\circ$ which is preserved by $\nabla$.
Define real subbundles $\mathcal{V}_{\mathbb{R}},\mathcal{N}_{\mathbb{R}}$ of $\mu^*T^{\mathbb{R}}B^\circ$ over $\mathcal{U}^\circ$ by
$$\mathcal{V}_{\mathbb{R}}=\{v+\ov{v}\ |\ v\in\mathcal{V}\}\subset \mu^*T^{\mathbb{R}}B^\circ,$$
and analogously for $\mathcal{N}_{\mathbb{R}}$. The bundle $\mathcal{V}_{\mathbb{R}}$ is isomorphic to $\mathcal{V}$ via the usual inverse map $T^{\mathbb{R}}B\to TB$ given by $u\mapsto \frac{u-iJ(u)}{2}$ (and similarly for $\mathcal{N}_{\mathbb{R}}$), and on $\mathcal{U}^\circ$ we have a splitting
\begin{equation}\label{splitta}
\mu^*T^{\mathbb{R}}B^\circ=\mathcal{V}_{\mathbb{R}}\oplus \mathcal{N}_{\mathbb{R}}.
\end{equation}

Consider now the Stein factorization of $\mu:\mathcal{U}\to B$, given by
$$\mathcal{U}\to Z\overset{p}{\to} B,$$
where $\mathcal{U}\to Z$ has connected fibers and $p:Z\to B$ is finite. Define also $Z^\circ:=p^{-1}(B^\circ)$. To complete the proof of Theorem \ref{form}, we will then need the following theorem which is implicit in the work of Hwang \cite{Hw}, and also appears in the recent work of Bakker-Schnell (\cite[Proposition 3.2 and proof of Theorem 1.1]{BS}) relying on ideas of Hwang \cite{Hw2,Hw}:

\begin{theorem}\label{bs} Suppose the splitting \eqref{splitta} is preserved by $\nabla^{\rm SK}$, then $p:Z\to B$ is an isomorphism.
\end{theorem}

We can now give the proof of Theorem \ref{form}:

\begin{proof}[Proof of Theorem \ref{form}]
Since $B\not\cong\mathbb{P}^n$, we have the nontrivial splitting \eqref{splitta}. By definition, $\mathcal{V}_{\mathbb{R}}$ is preserved by $J$, and since $\mathcal{V}$ is preserved by $\nabla$ (and $\nabla J=0$), it follows that $\mathcal{V}_{\mathbb{R}}$ is also preserved by $\nabla$.

We claim that $\mathcal{V}_{\mathbb{R}}$ is preserved by $\nabla^{\rm SK}$. To see this, recall that Freed shows in \cite[(1.29)]{Fr} that the special K\"ahler connection on $T^{\mathbb{R}}B$ is given by
\begin{equation}\label{freed}
\nabla^{\rm SK}=\nabla+A+\ov{A},
\end{equation}
where as usual $\nabla$ is the Levi-Civita connection of $g_{\rm SK}$ and
$A\in \Lambda^{1,0}\mathrm{Hom}(TB^\circ,\ov{TB^\circ})$ is given by
\begin{equation}\label{indiana}
A_{ij}^{\ov{\ell}}=\sqrt{-1}g_{\rm SK}^{k\ov{\ell}}\Xi_{ijk},\end{equation}
and the same holds for the pullback connection on $\mathcal{U}^\circ$.
Given a local section $\alpha$ of $\mathcal{V}$ and a local $(1,0)$ vector field $v\in T\mathcal{U}$, we wish to show that
$$\nabla^{\rm SK}_{v+\ov{v}} (\alpha+\ov{\alpha})\in\mathcal{V}_{\mathbb{R}}.$$
Since we know that $\nabla_{v+\ov{v}} (\alpha+\ov{\alpha})\in\mathcal{V}_{\mathbb{R}}$, it suffices to check that
$$(A+\ov{A})_{v+\ov{v}}(\alpha+\ov{\alpha})=A_v(\alpha)+\ov{A_v(\alpha)}\in\mathcal{V}_{\mathbb{R}},$$
and so it suffices to see that
$$A_v(\alpha)\in \ov{\mathcal{V}},$$
or equivalently that $g_{\rm SK}(A_v(\alpha), \ov{\zeta})=0$ for all local sections $\zeta$ of $\mathcal{N}$. But from \eqref{indiana} we see that
$$g_{\rm SK}(A_v(\alpha),\ov{\zeta})=\sqrt{-1} \Xi(v,\alpha,\zeta),$$
which vanishes by Theorem \ref{rigid} (c). This concludes the proof that $\mathcal{V}_{\mathbb{R}}$ is preserved by $\nabla^{\rm SK}$.  An analogous argument shows that
$\mathcal{N}_{\mathbb{R}}$ is also preserved by $\nabla^{\rm SK}$, and so the splitting \eqref{splitta} is preserved by $\nabla^{\rm SK}$. Applying Theorem \ref{bs} we see that $p:Z^\circ\to B^\circ$ is an isomorphism, so we may assume that $\mu:\mathcal{U}^\circ\to B^\circ$ has connected fibers. The vector bundle $\mu^*T^{\mathbb{R}}B^\circ$ is trivial when restricted to these fibers, and its subbundles $\mathcal{V}_{\mathbb{R}},\mathcal{N}_{\mathbb{R}}$ restricted to a fiber are preserved by the pullback connection $\nabla^{\rm SK}$ (which when restricted to the fiber is a trivial connection), and so $\mathcal{V}_{\mathbb{R}}$ and $\mathcal{N}_{\mathbb{R}}$ are pullbacks of vector bundles on $B^\circ$ (denoted by the same notation), which are subbundles of $T^{\mathbb{R}}B^\circ$ and are still preserved by $\nabla^{\rm SK}$.

We then define a $(1,1)$-form $\psi$ on $B^\circ$ by projecting $\omega_{\rm SK}$ onto $\mathcal{V}_{\mathbb{R}}$. Since $\nabla^{\rm SK}\omega_{\rm SK}=0$ and $\mathcal{V}_{\mathbb{R}}$ is preserved by $\nabla^{\rm SK}$, it follows that $\nabla^{\rm SK}\psi=0$ (and also $\nabla\psi=0$ for the same reason), and since $\mathcal{V}_{\mathbb{R}}$ is a nontrivial proper subbundle of $T^{\mathbb{R}}B^\circ$, we see that $\psi$ is nonzero and not proportional to $\omega_{\rm SK}$, and we are done.
\end{proof}

\section{Comments about the case when \texorpdfstring{$B$}{B} is singular}\label{se8}
It is tempting to ask whether our method can be used to prove that $B\cong\mathbb{P}^n$ even when $B$ is singular. As mentioned in the Introduction, this is currently known only for $n\leq 2$ \cite{BK,HX,Ou}.
In general, it is known that $B$ is a normal projective variety, with at worst klt singularities, which is Fano with Picard number one. The natural generalization of our approach (following \cite{Si}, who generalized Mori's Theorem \cite{Mor} to the singular setting) would be to consider a functorial resolution of singularities $\pi:\ti{B}\to B$ and to show that we must have $\ti{B}\cong\mathbb{P}^n$, which forces $B\cong\mathbb{P}^n$ as well. In this setting, $\ti{B}$ is a uniruled projective manifold and $\ti{D}=\pi^{-1}(D)$ is a divisor, so many of our arguments above can be repeated on $\ti{B}^\circ:=\ti{B}\backslash\ti{D}$, which carries a special K\"ahler metric $\omega_{\rm SK}$. The fact that $\pi$ is functorial gives us a morphism $\mu:\pi^*TB\to T\ti{B}$ which is an isomorphism on $\ti{B}^\circ$, where $TB=\mathrm{Hom}(\Omega^1_B,\mathcal{O}_B)$ is the reflexive tangent sheaf.
Given a rational curve $\nu:\mathbb{P}^1\to\ti{B}$, which is not contained in $\ti{D}$, pulling back $\mu$ via $\nu$ we obtain a sheaf injection $$\mathcal{A}:=(\pi\circ\nu)^{[*]}TB\to \nu^* T\ti{B},$$
between these vector bundles on $\mathbb{P}^1$ (which both split as a direct sum of line bundles which should have nonnegative degrees). Here we use the standard reflexive pullback notation $(\pi\circ\nu)^{[*]}TB:=(\nu^*\pi^*TB)^{**}$. Using Theorem \ref{unir}, if $\ti{B}\not\cong\mathbb{P}^n$ then
$\nu^* T\ti{B}$ contains a nontrivial $\mathcal{O}$ factor, hence so does $\mathcal{A}$. To implement our strategy, one would need a rigidity statement like in Theorem \ref{rigid} for either one of these trivial summands, and a crucial ingredient of the proof of the rigidity statement is that sections of the dual of the relevant bundle should have bounded norm (with respect to the pullback of $\omega_{\rm SK}$). The first fundamental issue is that  it is not clear to us how to show that sections of  $\nu^*\Omega^1_{\ti{B}}$ or of $\mathcal{A}^*$ have bounded norm. The key ingredient for this when $B$ is smooth was the estimate \eqref{sl}, but when $B$ is singular this by itself is not sufficient to prove boundedness.

What can be shown using results in \cite{GKP} is rather that sections of the reflexive pullback $(\pi\circ\nu)^{[*]}\Omega^{[1]}_B$ have bounded norm, but in general the Grothendieck decomposition of this vector bundle is different from those of $\nu^*\Omega^1_{\ti{B}}$ and $\mathcal{A}^*$, and it may happen that these have some nontrivial $\mathcal{O}$ factor but $(\pi\circ\nu)^{[*]}\Omega^{[1]}_B$ does not, which invalidates our approach. This undesirable phenomenon can only happen when the generic rational curve (of the type that we are considering) when projected down to $B$ always passes through some singular point of $B$. This however seems unavoidable in general, as finding low-degree rational curves in normal Fano varieties that can be deformed to avoid the singularities is a very delicate problem in algebraic geometry, see e.g. \cite{KM,Kol,Xu}.


\begin{thebibliography}{99}
\bibitem{Ba} B. Bakker, {\em A short proof of a conjecture of Matsushita}, preprint, arXiv:2209.00604.
\bibitem{BS} B. Bakker, C. Schnell, {\em A Hodge-theoretic proof of Hwang's theorem on base manifolds of Lagrangian fibrations}, preprint, arXiv:2311.08977.
\bibitem{Bar} D. Barlet, {\em D\'eveloppement asymptotique des fonctions obtenues par int\'egration sur les fibres}, Invent. Math. {\bf 68} (1982), no. 1, 129--174.
\bibitem{BM} A. Bayer, E. Macr\`i, {\em MMP for moduli of sheaves on K3's via wall-crossing: nef and movable cones, Lagrangian fibrations}. Invent. Math. {\bf 198} (2014), no. 3, 505--590.
\bibitem{Be} A. Beauville, {\em Vari\'et\'es K\"ahleriennes dont la premi\`ere classe de Chern est nulle}, J. Differential Geom. {\bf 18} (1983), no. 4, 755--782.
\bibitem{BK} F. Bogomolov, N. Kurnosov, {\em Lagrangian fibrations for IHS fourfolds}, preprint, arXiv:1810.11011.
\bibitem{BJ} S. Boucksom, M. Jonsson, {\em Tropical and non-Archimedean limits of degenerating families of volume forms}, J. \'Ec. polytech. Math. {\bf 4} (2017), 87--139.
\bibitem{CH} M. Callies, A. Haydys, {\em  Local models of isolated singularities for affine special K\"ahler structures in dimension two}, Int. Math. Res. Not. IMRN 2020, no. 17, 5215--5235.
\bibitem{Cam} F. Campana, {\em Local projectivity of Lagrangian fibrations on hyperk\"ahler manifolds}, Manuscripta Math. {\bf 164} (2021), no. 3-4, 589--591.
\bibitem{CMSB} K. Cho, Y. Miyaoka, N.I. Shepherd-Barron, {\em Characterizations of projective space and applications to complex symplectic manifolds} in {\em  Higher dimensional birational geometry (Kyoto, 1997)}, 1--88, Adv. Stud. Pure Math., 35, Math. Soc. Japan, Tokyo, 2002.
\bibitem{De} J.-P. Demailly, {\em Complex analytic and differential geometry}, online book.
\bibitem{DW} R. Donagi, E. Witten, {\em Supersymmetric Yang-Mills theory and integrable systems}, Nuclear Phys. B {\bf 460} (1996), no. 2, 299--334.
\bibitem{Fr}   D. Freed,  {\em Special K\"ahler Manifolds},  Comm. Math. Phys. {\bf 203} (1999), no. 1, 31--52.
\bibitem{GR} H. Grauert, R. Remmert, {\em Plurisubharmonische Funktionen in komplexen R\"aumen}, Math. Z. {\bf 65} (1956), 175--194.
\bibitem{GKP} D. Greb, S. Kebekus, T. Peternell, {\em Reflexive differential forms on singular spaces. Geometry and cohomology}, J. Reine Angew. Math. {\bf 697} (2014), 57--89.
\bibitem{GL} D. Greb, C. Lehn, {\em Base manifolds for Lagrangian fibrations on hyperk\"ahler manifolds},  Int. Math. Res. Not. IMRN 2014, no. 19, 5483--5487.
\bibitem{GTZ} M. Gross, V. Tosatti, Y. Zhang, {\em Collapsing of abelian fibred Calabi-Yau manifolds}, Duke Math. J. {\bf 162} (2013), no. 3, 517--551.
\bibitem{GTZ2} M. Gross, V. Tosatti, Y. Zhang,  {\em Gromov-Hausdorff collapsing of Calabi-Yau manifolds},  Comm. Anal. Geom. {\bf 24} (2016), no. 1, 93--113.
\bibitem{GTZ3} M. Gross, V. Tosatti, Y. Zhang,  {\em Geometry of twisted K\"ahler-Einstein metrics and collapsing}, Comm. Math. Phys. {\bf 380} (2020), no. 3, 1401--1438.
\bibitem{Har} R. Hartshorne, {\em Algebraic geometry}, Graduate Texts in Mathematics, No. 52. Springer--Verlag, New York-Heidelberg, 1977.
\bibitem{Ha} A. Haydys, {\em  Isolated singularities of affine special K\"ahler metrics in two dimensions}, Comm. Math. Phys. {\bf 340} (2015), no. 3, 1231--1237.
\bibitem{HT} H.-J. Hein, V. Tosatti, {\em Remarks on the collapsing of torus fibered Calabi-Yau manifolds},  Bull. Lond. Math. Soc. {\bf 47} (2015), no. 6, 1021--1027.
\bibitem{Her} C. Hertling, {\em $tt^*$ geometry, Frobenius manifolds, their connections, and the construction for singularities}, J. Reine Angew. Math. {\bf 555} (2003), 77--161.
\bibitem{HM} D. Huybrechts, M. Mauri, {\em Lagrangian fibrations}, Milan J. Math. {\bf 90} (2022), no. 2, 459--483.
\bibitem{HX} D. Huybrechts, C. Xu, {\em Lagrangian fibrations on hyperk\"ahler fourfolds},  J. Inst. Math. Jussieu {\bf 21} (2022), no. 3, 921--932.
\bibitem{Hw2} J.-M. Hwang, {\em Deformation of holomorphic maps onto Fano manifolds of second and fourth Betti numbers $1$}, Ann. Inst. Fourier (Grenoble) {\bf 57} (2007), no. 3, 815--823.
\bibitem{Hw} J.-M. Hwang, \emph{Base manifolds for fibrations of projective irreducible symplectic manifolds}, Invent. Math. {\bf 174} (2008),  no. 3, 625--644.
\bibitem{Hw3} J.-M. Hwang, {\em Mori geometry meets Cartan geometry: varieties of minimal rational tangents}, in {\em Proceedings of the International Congress of Mathematicians--Seoul 2014. Vol. 1}, 369--394, Kyung Moon Sa, Seoul, 2014.
\bibitem{HO} J.-M. Hwang, K. Oguiso, {\em Characteristic foliation on the discriminant hypersurface of a holomorphic Lagrangian fibration}, Amer. J. Math. {\bf 131} (2009), no. 4, 981--1007.
\bibitem{KM} S. Keel, J. McKernan, {\em  Rational curves on quasi-projective surfaces}, Mem. Amer. Math. Soc. {\bf 140} (1999), no. 669, viii+153 pp.
\bibitem{Ki} D. Kim, {\em Canonical bundle formula and degenerating families of volume forms}, preprint, arXiv:1910.06917.
\bibitem{Koba} S. Kobayashi, {\em Differential geometry of complex vector bundles}, Princeton University Press, 1987.
\bibitem{Kol} J. Koll\'ar, {\em Cone theorems and bug-eyed covers}, J. Algebraic Geom. {\bf 1} (1992), no. 2, 293--323.
\bibitem{Kol2} J. Koll\'ar, {\em Rational curves on algebraic varieties}, Ergebnisse der Mathematik und ihrer Grenzgebiete, 32. Springer-Verlag, Berlin, 1996.
\bibitem{Lu} Z.  Lu, {\em  A note on special K\"{a}hler manifolds},  Math. Ann. {\bf 313} (1999), no. 4, 711--713.
\bibitem{Mar} E. Markman, {\em Lagrangian fibrations of holomorphic-symplectic varieties of $K3^{[n]}$-type}, in {\em  Algebraic and complex geometry}, 241--283, Springer Proc. Math. Stat., 71, Springer, Cham, 2014.
\bibitem{Ma} D. Matsushita, {\em On fibre space structures of a projective irreducible symplectic manifold}, Topology {\bf 38} (1999), no. 1, 79--83; Addendum {\bf 40} (2001), no. 2, 431--432.
\bibitem{Ma3} D. Matsushita, {\em Holomorphic symplectic manifolds and Lagrangian fibrations}, in {\em Monodromy and differential equations (Moscow, 2001)}, Acta Appl. Math. {\bf 75} (2003), no. 1-3, 117--123.
\bibitem{Ma2} D. Matsushita, {\em Higher direct images of dualizing sheaves of Lagrangian fibrations}, Amer. J. Math. {\bf 127} (2005), no. 2, 243--259.
\bibitem{Ma5} D. Matsushita, {\em On deformations of Lagrangian fibrations}, in {\em $K3$ Surfaces and their moduli}, 237--243, Birkh\"auser, Cham, 2016.
\bibitem{Ma4} D. Matsushita, {\em On isotropic divisors on irreducible symplectic manifolds}, in {\em Higher dimensional algebraic geometry--in honour of Professor Yujiro Kawamata's sixtieth birthday}, 291--312, Adv. Stud. Pure Math., 74, Math. Soc. Japan, Tokyo, 2017.
\bibitem{MM} Y. Miyaoka, S. Mori, {\em A numerical criterion for uniruledness}, Ann. of Math. (2) {\bf 124} (1986), no. 1, 65--69.
\bibitem{Mok} N. Mok, {\em The uniformization theorem for compact K\"ahler manifolds of nonnegative holomorphic bisectional curvature}, J. Differential Geom. {\bf 27} (1988), no. 2, 179--214.
\bibitem{Mor} S. Mori, {\em Projective manifolds with ample tangent bundles}, Ann. of Math. (2) {\bf 110} (1979), no. 3, 593--606.
\bibitem{Nam} Y. Namikawa, {\em Projectivity criterion of Moishezon spaces and density of projective symplectic varieties}, Internat. J. Math. {\bf 13} (2002), no. 2, 125--135.
\bibitem{Nag} Y. Nagai, {\em Dual fibration of a projective Lagrangian fibration}, unpublished preprint, 2005.
\bibitem{Oh} T. Ohsawa, {\em Analysis of several complex variables}, American Mathematical Society, Providence, RI, 2002.
\bibitem{Ou} W. Ou, {\em Lagrangian fibrations on symplectic fourfolds}, J. Reine Angew. Math. {\bf 746} (2019), 117--147.
\bibitem{Sch} W. Schmid, {\em Variation of Hodge structure: the singularities of the period mapping}, Invent. Math. {\bf 22} (1973), 211--319.
\bibitem{SY} J. Shen, Q. Yin, {\em Topology of Lagrangian fibrations and Hodge theory of hyper-K\"ahler manifolds. With Appendix B by Claire Voisin}, Duke Math. J. {\bf 171} (2022), no. 1, 209--241.
\bibitem{Si} P. Sieder, {\em Varieties with ample tangent sheaves}, Manuscripta Math. {\bf 157} (2018), no. 1-2, 257--261.
\bibitem{SY2} Y.-T. Siu, S.-T. Yau, {\em Compact K\"ahler manifolds of positive bisectional curvature}, Invent. Math. {\bf 59} (1980), no. 2, 189--204.
\bibitem{ST} J. Song, G. Tian, {\em Canonical measures and K\"ahler-Ricci flow}, J. Amer. Math. Soc. {\bf 25} (2012), no. 2, 303--353.
\bibitem{Taka} S. Takayama, {\em Simple connectedness of weak Fano varieties}, J. Algebraic Geom. {\bf 9} (2000), no. 2, 403--407.
\bibitem{Taka2} S. Takayama, {\em Asymptotic expansions of fiber integrals over higher-dimensional bases}, J. Reine Angew. Math. {\bf 773} (2021), 67--128.
\bibitem{To} V. Tosatti, \emph{Adiabatic limits of Ricci-flat K\"ahler metrics}, J. Differential Geom. {\bf 84} (2010), no. 2, 427--453.
\bibitem{TZ3} V. Tosatti, Y. Zhang, {\em Finite time collapsing of the K\"ahler-Ricci flow on threefolds}, Ann. Sc. Norm. Super. Pisa Cl. Sci. {\bf 18} (2018), no.1, 105--118.
\bibitem{TZ} V. Tosatti, Y. Zhang, {\em Collapsing hyperk\"ahler manifolds}, Ann. Sci. \'Ec. Norm. Sup\'er. {\bf 52} (2020), no. 3, 751--786.
\bibitem{vV} B. van Geemen, C. Voisin, {\em On a conjecture of Matsushita},  Int. Math. Res. Not. IMRN 2016, no. 10, 3111--3123.
\bibitem{Vo} C. Voisin, {\em Torsion points of sections of Lagrangian torus fibrations and the Chow ring of hyper-K\"ahler manifolds}, in {\em  Geometry of moduli}, 295--326, Abel Symp., 14, Springer, Cham, 2018.
\bibitem{Xu} C. Xu, {\em Strong rational connectedness of surfaces}, J. Reine Angew. Math. {\bf 665} (2012), 189--205.
\bibitem{Yosh} K.-I. Yoshikawa, {\em On the boundary behavior of the curvature of $L^2$-metrics}, preprint, arXiv:1007.2836.
\bibitem{Yos} K. Yoshioka, {\em Bridgeland's stability and the positive cone of the moduli spaces of stable objects on an abelian surface}, in {\em Development of moduli theory--Kyoto 2013}, 473--537,
Adv. Stud. Pure Math., 69, Math. Soc. Japan, Tokyo, 2016.
\end{thebibliography}
\end{document}